\documentclass[11pt]{amsart}

\usepackage{amsmath, amssymb, amsthm, amsfonts, mathrsfs}

\input xy
\xyoption{all}

\newtheorem{prop}{Proposition}[section]
\newtheorem{defn}[prop]{Definition}
\newtheorem{lemma}[prop]{Lemma}
\newtheorem{remark}[prop]{Remark}
\newtheorem{theorem}[prop]{Theorem}
\newtheorem{cor}[prop]{Corollary}
\theoremstyle{definition}

\theoremstyle{definition}
\newtheorem{example}[prop]{Example}
\newtheorem{claim}[prop]{Claim}
\newcommand{\infinity}{\ensuremath{\infty } }

\begin{document}

\title{Dynamics of Irreducible Endomorphisms of $F_n$}

\author{Patrick Reynolds}

\address{\tt Department of Mathematics, University of Illinois at
  Urbana-Champaign, 1409 West Green Street, Urbana, IL 61801, USA}
  \email{\tt
  preynol3@math.uiuc.edu} 

\date{\today}

\maketitle

\begin{abstract}
We consider the class non-surjective irreducible endomorphisms of the free group $F_n$.  We show that such an endomorphism $\phi$ is topologically represented by a simplicial immersion $f:G \rightarrow G$ of a marked graph $G$; along the way we classify the dynamics of $\partial \phi$ acting on $\partial F_n$: there are at most $2n$ fixed points, all of which are attracting.  After imposing a necessary additional hypothesis on $\phi$, we consider the action of $\phi$ on the closure $\overline{CV}_n$ of the Culler-Vogtmann Outer space.  We show that $\phi$ acts on $\overline{CV}_n$ with ``sink'' dynamics: there is a unique fixed point $[T_{\phi}]$, which is attracting; for any compact neighborhood $N$ of $[T_{\phi}]$, there is $K=K(N)$, such that $\overline{CV}_n\phi^{K(N)} \subseteq N$.  The proof uses certian projections of trees coming from invariant length measures.  These ideas are extended to show how to decompose a tree $T$ in the boundary of Outer space by considering the space of invariant length measures on $T$; this gives a decomposition that generalizes the decomposition of geometric trees coming from Imanishi's theorem.
\end{abstract}

\section{Introduction}

In what follows, $F_n$ denotes the rank $n$ free group; $\overline{CV}_n$ denotes the ``Thurston compactification'' of the Culler-Vogtmann Outer space; and $\overline{cv}_n$ denotes the space of very small actions of $F_n$ on $\mathbb{R}$-trees, so $\overline{CV}_n$ consists of projective classes $[T]$ of trees $T \in \overline{cv}_n$; see Section \ref{S.CVn} for definitions.

In \cite{BH92} Bestvina and Handel introduce the notion of an \emph{irreducible} outer automorphism of the free group $F_n$ (see Section \ref{S.Irred}); this class of (outer) automorphisms serves as an analog of the class of (mapping classes of) pseudo-Anosov diffeomorphisms of a hyperbolic surface.  Bestvina and Handel introduce \emph{train track represetatives} for irreducible elements of $Out(F_n)$; these are topological representatives that allow for very close control over rates of growth of conjugacy classes.  It is shown in \cite{BH92} that any irreducible outer automorphism of $F_n$ has a train track representative; see Section \ref{SS.TrainTracks}.

The terminiology of train tracks in \cite{BH92} seems to anticipate the work in \cite{BFH97}, where to each irreducible automorphism of $F_n$ there is associated a pair of abstract laminations.  These abstract laminations on $F_n$ are formalized by Coulbois, Hilion, and Lustig in \cite{CHL08a} via the notion of an \emph{algebraic lamination} (see Subsection \ref{SS.Laminations}); this formalism, as well as its applications in \cite{CHL08b, CHL08c}, does well to compliment the theory of train tracks for studying free group (outer) automorphisms.  The algebraic laminations associated to a free group automorphism are analogous to the geodesic surface laminations associated to a pseudo-Anosov surface diffeomorphism: they are a sort of asymptotic invariant encoding limits of iterates of the automorphism (and its inverse) on primitive elements--the free group analogs of essential simple closed curves on a surface.  Generalizations of the tools of \cite{BFH97} and \cite{BH92} were used by Bestvina-Feighn-Handel to prove the Tits Alternative for $Out(F_n)$ \cite{BFH00, BFH05}.

Inspired by the applicability of these dynamical techniques for understanding elements of $Out(F_n)$, we study non-surjective irreducible (outer) endomorphisms of $F_n$ from a dynamical viewpoint.  An endomorphism $\phi:F_n \rightarrow F_n$ is \emph{irreducible} if no power of $\phi$ maps a non-trivial, proper free factor of $F_n$ into a conjugate of itself, and if this condition holds for any power of $\phi$ as well; see Section \ref{S.Irred}.  

Suppose that $\phi:F_n \rightarrow F_n$ is irreducible; it follows from work of Bestvina-Feighn-Handel \cite{BFH97} that one may associate to $\phi$ an algebraic lamination $\Lambda_{\phi}$ and a (projective) \emph{stable tree} $[T_{\phi}] \in \overline{CV_n}$; see Subsections \ref{SS.StableLam} and \ref{SS.TrainTracks}.  There is a natural right action of $\phi$ on the set of $\mathbb{R}$-trees, equipped with minimal, isometric actions of $F_n$, and $[T_{\phi}]$ has the property that $[T_{\phi}\phi]=[T_{\phi}]$.   Using the techniques of \cite{BFH97}, Coulbois-Hilion have shown that for irreducible $\alpha \in Out(F_n)$, the stable tree $T_{\alpha}$ has a strong mixing property--it is \emph{indecomposable} \cite{CH10}; see Definition \ref{D.Indecomposable}. In constrast with the case of outer automorphisms, we obtain (Proposition \ref{P.FreeSimp}):

\begin{prop}
Let $\phi:F_n \rightarrow F_n$ an irreducible endomorphism that is non-surjective.  There is a free simplicial $F_n$-tree $T_{\phi}$ such that $[T_{\phi}]\phi=[T_{\phi}]$.
\end{prop}

This immediately gives (Corollary \ref{C.NoIllegalTurns}):

\begin{cor}
Let $\phi:F_n \rightarrow F_n$ an irreducible endomorphism that is non-surjective.  Then (the outer class of) $\phi$ is topologically-represented by a train track map with no illegal turns. 
\end{cor}

Building on the techniques of \cite{BFH97}, Levitt and Lustig show in \cite{LL03} that any irreducible outer automorphism of $F_n$ acts on $\overline{CV}_n$ with north-south dynamics: there are exactly two fixed points, one attracting and one repelling, such that convergence to the attractor is uniform on compact subsets avoiding the repeller.

Unlike in the case of $Out(F_n)$ one needs to impose an additional assumption on a non-surjective irreducible endomorphism $\phi$ to ensure that the action of $\phi$ on the set of $F_n$-trees induces an action on $\overline{CV_n}$; we call such $\phi$ \emph{admissible} (see Section \ref{S.Admissible}).  In this case we consider the dynamics of the action of $\phi$ on $\overline{CV_n}$; we show (Theorem \ref{T.Main}):

\begin{theorem}\label{T.Intro1}
Let $\phi:F_n \rightarrow F_n$ be an admissible irreducible endomorphism that is non-surjective.  Then $\phi$ has a unique fixed point $[T_\phi] \in \overline{CV_n}$, which is free and simplicial; for any $[T] \neq [T_{\phi}]$ one has that $[T]\phi^k \rightarrow [T_{\phi}]$; and for any compact nieghborhood $N$ of $[T_{\phi}]$, there is $k=k(N)$ such that $\overline{CV}_n \phi^k \subseteq N$.
\end{theorem}

It should be noted that Theorem \ref{T.Intro2} is novel in the sense that $\phi$ is not assumed to be invertible.  This result turns out to be much more difficult to prove than North-South dynamics for irreducible automorphisms of $F_n$ \cite{LL03}, which is in turn much more difficult to prove than North-South dynamics for pseudo-Anosov surface automorphisms.  The latter two results use ``backwards iteration'' in an essential way, and it is reasonable to say that many of the complications in the present work stem from the lack of an inverse.    

Theorem \ref{T.Intro1} implies that for $k$ sufficiently large, the subgroups $\phi^k(F_n)$ have a strong rigidity property (Corollary \ref{C.Rigid}), which seems interesting to us:

\begin{cor}
 For any $C>1$, there is a finitely generated, non-abelian subgroup $H \leq F_n$, such that for any non-trivial $h,h' \in H$ and any trees $T,T' \in \overline{cv}_n$, one has $l_T(h)>0$ and $$\frac{1}{C} \leq \frac{l_T(h)/l_T(h')}{l_{T'}(h)/l_{T'}(h')} \leq C$$
\end{cor}

Along the way to proving Theorem \ref{T.Intro1}, we introduce the rudiments of a decomposition theory of individual trees $T \in \overline{cv}_n$.  At the heart of this approach is a study of the space $M_0(T)$ of invariant, non-atomic \emph{length measures} on $T$ (Section \ref{S.Measures}); these objects, introduced by Paulin, generalize measured laminations on surfaces.  In \cite{Gui00} Guirardel uses length measures as part of an approach to study the dynamics of $Out(F_n)$ acting on the boundary of $\overline{CV}_n$; there he shows that for $T \in \overline{cv}_n$ with dense orbits, the projectivization of $M_0(T)$ embeds in the boundary of $\overline{CV}_n$.  This shows, in particular, that the space $M_0(T)$ is finite dimensional.  

We now briefly recall the structure of trees dual to measured laminations on surfaces, in order to contrast with the trees in $\overline{cv}_n$.  Let $\mathscr{L}=(L,\mu)$ be a measured lamination on a surface $S$, and let $T=T_{\mathscr{L}}$ denote the dual tree; see, for example, \cite{MS84} or \cite{MKap01}.  If $L$ is not minimal, then $T$ can be decomposed in a way that parallels the decompositon of $L$ into minimal components--$T$ is a \emph{graph of actions}; see Section \ref{S.GoA}.  A feature of (minimal) \emph{arational measured laminations} is that every half-leaf is dense; this implies that a tree dual to an arational measured laminiation is indecomposable.  It follows that, if $L$ has no compact leaf, then either $T$ is indecomposable, or $T$ splits as a graph of indecomposable actions.  

The structure of some trees in $\overline{cv}_n$ is divergent from this picture: there are trees $T \in \overline{cv}_n$ such that $T$ is neither indecomposable nor a graph of actions; see Example \ref{E.ExceptionalSet}.  There is a holonomy pseudogroup associated to $T$, which is completely analogous to the holonomy pseudogroup associated to a lamination, and this psuedogroup contains an \emph{exceptional} set, in contrast with the surface case; see Section \ref{S.Measures}.  To understand the dynamical structure of $T$, it is useful to consider certain \emph{projections} of $T$; for this approach it is critical that the dynamics of the action $F_n \curvearrowright T$ are ``visible'' to length measures.  This is accomplished via the following (Proposition \ref{P.ExceptionalMeasure}):

\begin{prop}\label{P.Intro1}
 Let $T \in \overline{cv}_n$ have dense orbits.  Suppose that $T$ does not split as a graph of actions and that $T$ contains an exceptional set $X$.  Then there is $\mu \in M_0(T)$ supported on $X$.  
\end{prop}

A non-empty $F_n$-invariant subset $X \subseteq T$ is called \emph{exceptional} if for any finite subtree $K \subseteq T$, $X \cap K$ is empty, finite, or a Cantor set with finitely many points added.  According to Proposition \ref{P.Intro1}, given a tree $T \in \overline{cv}_n$ with dense orbits, such that $T$ is not a graph of actions and such that $T$ contains an exceptional set $X$, we can find an invariant measure $\mu$ supported on $X$.  We may then pass to a \emph{projection} of $T$: the measure $\mu$ gives rise to a pseudometric $d_{\mu}$ on $T$, where $d_{\mu}(x,y):=\mu([x,y])$.  Making this pseudometric Hausdorff gives a tree $T_{\mu}$, equipped with an isometric action of $F_n$ (see \cite{Gui00, CHL07}); in a precise sense, the action $F_n \curvearrowright T_{\mu}$ isolates the dynamics of $F_n \curvearrowright X$.


Note that an $F_n$-tree $T$ can be non-uniquely ergodic, even if it has strong mixing properties: examples already come from non-uniquely ergodic, arational laminations on surfaces \cite{KN76}.  Defnine a partial order $\leq$ on $M_0(T)$ via $\mu \leq \mu'$ if and only if $Supp(\mu) \subseteq Supp(\mu')$; this gives an equivalence with classes $[[\mu]] =\{ \nu \in M_0(T)| Supp(\nu)=Supp(\mu)\}$, which serve as candidates for the ``components'' of $T$.  Indeed, in the case that $T$ is dual to a measured lamination $(\lambda, \mu_0)$ on a surface with boundary, the set of $[[.]]$-classes of invariant (non-atomic) length measures on $T$ bijectively corresponds to the set of sublaminations of $\lambda$.

Having understood that some dynamical structure of $T$ is clarified by considering projections of $T$, we complete our analogy with measured laminations by associating to every $[[.]]$-class of ergodic measures in $M_0(T)$ a canonical \emph{mixing} action (Definition \ref{D.Mixing}, Proposition \ref{P.MixingSubaction}, and Corollary \ref{C.UniqueMaxMix}).  Below we give a simplified form of our decomposition result (Theorem \ref{T.Decompose}).




\begin{theorem}\label{T.Intro2}
Let $T \in \overline{cv}_n$ have dense orbits, and let $\{\nu_1,...,\nu_r\}$ be a maximal set of mutually-singular ergodic measures.  
 \begin{enumerate}
  \item [(i)] For each $\nu_i$ with non-degenerate support, there is associated to $[[\nu_i]]$ a mixing action $H([[\nu_i]]) \curvearrowright T([[\nu_i]])$, such that $T([[\nu_i]])$ is unique up to translation in $T$,
  \item [(ii)] For each $\nu_j$ with degenerate support, there is a projection $T \rightarrow T_{[[\nu_j]]}$, such that:
    \begin{enumerate}
     \item [(a)] $dim(M_0(T_{[[\nu_j]]}))<dim(M_0(T))$,
     \item [(b)] for all $\nu_j' \in [[\nu_j]]$, $M_0(T_{\nu_j'})$ is naturally identified with $M_0(T_{[[\nu_j]]})$,
     \item [(c)] for all $\nu_j' \in [[\nu_j]]$, $L^2(T_{\nu_j'})=L^2(T_{[[\nu_j]]})$
    \end{enumerate}
 \end{enumerate}
\end{theorem}

In the statement, $M_0(T)$ denotes the space of non-atomic invariant length measures on $T$.  The subtree $T([[\nu_i]]) \subseteq T$ is ``equal to'' the support set $Supp(\nu_i)$, and $H([[\nu_i]])$ is the setwise stabilizer of $T([[\nu_i]])$.  In short, part (i) of Theorem 2 is analogous to the usual dynamical decomposition of measured surface laminations, and, more generally, the decomposition of \emph{geometric trees} coming from Imanishi's Theorem; see \cite{GLP94}.  So, part (ii) illustrates only \emph{non-geometric} phenomena.  

Theorem \ref{T.Intro2} is related to ongoing work (in preparation) of Guirardel and Levitt about actions of finitely presented groups on R-trees \cite{GL10}; see the Acknowledgements, as well as Sections \ref{S.Measures} and \ref{S.Decompose} for further discussion.  


\subsection{Organization}

In Section \ref{S.Back} we collect basic background material about $\mathbb{R}$-trees, Outer space, algebraic laminations, and the observers' topology on trees; this section is expository, except for Lemma \ref{L.Observers}.  

Section \ref{S.Irred} introduces topological representatives, train tracks and stable trees; there we introduce \emph{expansive} (Definition \ref{D.Expansive}) endomorphisms and show (Proposition \ref{P.Dichot}) that an irreducible endomorphism is either expansive or an automorphism; this gives, via Corollary \ref{C.TisFree}, that stable trees of non-surjective irreducible endomorphisms are free.

Section \ref{S.GoA} is devoted to defining and giving basic properties of \emph{graphs of actions} (Definition \ref{D.GoA}), which are used in Section \ref{S.StableTree} to show (Proposition \ref{P.FreeSimp}) that the stable tree of a non-surjective irreduicble endomorphism is free and simplicial.  This immediately implies that such an endomorphism is topologically represented by a simplicial immersion (Corollary \ref{C.NoIllegalTurns}).  The proof of Proposition \ref{P.FreeSimp} shows (Corollary \ref{C.BoundaryDynamics}) that for $\phi:F_n \rightarrow F_n$ irreducible and non-surjective, $\partial \phi$ acts on $\partial F_n$ with finitely-many fixed points, all of which are attracting.

We then turn to the question of the dynamics of $\phi$ acting on Outer space, denoted $CV_n$, and its closure, denoted $\overline{CV}_n$.  Section \ref{S.CVnDynamics} introduces the stable lamination, denoted $\Lambda_{\phi}$, associated to $\phi$.  We then state the convergenge criterion that will be used for the sequel: Proposition \ref{P.ConvergeLemma}, which is due to Bestvina-Feighn-Handel.  This immediately gives that $\phi$ acts on $CV_n$ with precisely one (attracting) fixed point--the stable tree of $\phi$.

Before proceeding to consider the dynamics of $\phi$ acting on $\overline{CV}_n$, we must impose a condition to ensure that $\phi$ acts on $\overline{CV}_n$; we explain what can go wrong in Section \ref{S.Admissible}, and then give the definition for \emph{admissible} endomorphisms (Definition \ref{D.Admissible}).  So that we may apply our convergence criterion, we show that being admissible is equivalent to a condition on $\Lambda_{\phi}$, namely that no leaf of $\Lambda_{\phi}$ is carried by a vertex group of a \emph{very small splitting} of $F_n$ (Proposition \ref{P.LamLemma}).  

Using Proposition \ref{P.LamLemma} along with Proposition \ref{P.ConvergeLemma}, in Section \ref{S.SimpGraphConverge} we get convergence for trees in the boundary of Outer space that split as a non-trivial graph of actions.   In Section \ref{S.IndecompConverge} we handle convergence for \emph{indecomposable} (Definition \ref{D.Indecomposable}) trees via a result (Proposition \ref{P.IndecompConverge}) of the author from \cite{R10a}.  

Section \ref{S.Measures} introduces the relevant measure theory on trees; we explain, following Guirardel, how measures on a tree $T$ encode morphisms $T \rightarrow T'$.  We state the key result of Guirardel (Proposition \ref{P.M0T}), that for $T$ with dense orbits, the space $M_0(T)$ of invariant non-atomic measures on $T$ is finite dimensional.  Next, we define \emph{exceptional sets} (Definition \ref{D.Exceptional}) and provide Example \ref{E.ExceptionalSet} to show that such things actually occur.  We present an iterative procedure (see Remark \ref{R.GLProcedure}), due to Guirardel-Levitt \cite{GL10}, for building \emph{transverse families} (Definition \ref{D.TF}) of subtrees.  In Subsection \ref{SS.ExceptionalMeasure}, we prove the critical result (Proposition \ref{P.ExceptionalMeasure}) for the rest of the paper: if $T$ is a tree in the boundary of Outer space with dense orbits, and if $T$ does not split as a graph of actions, then any exceptional set in $T$ supports an invaraint measure.   

In Section \ref{S.MainDynamics} we combine Proposition \ref{P.ExceptionalMeasure} with Lemma \ref{L.Observers} to get convergence for the remainder of actions in $\overline{CV}_n$.  The dynamics of $\phi$ acting on $\overline{CV}_n$ (Theorem \ref{T.Main}) then easily follows.  

Section \ref{S.Decompose} elaborates upon the measure-theoretic techniques introduced in Section \ref{S.Measures} to present an approach to decomposing trees in the boundary of Outer space.  For the remainder of the paper, we consider a tree $T$ with dense orbits.  In Subsection \ref{SS.F}, we define a transverse family $\mathscr{F}$ that gives a coarse decomposition of $T$ (Lemma \ref{L.SuppGraph}).  We then bring Proposition \ref{P.MixingSubaction} and Corollary \ref{C.UniqueMaxMix} to show how to associate to every invariant measure on $T$ a canonical mixing action; these actions are ``building blocks'' of $T$.  We collect the results of Section \ref{S.Decompose} to give our decomposition result, Theorem \ref{T.Decompose}.

\vspace{.25cm}
{\bf Acknowledgements:}  Thanks go to Thierry Coulbois, Arnaud Hilion, and Martin Lustig for many helpful conversations and comments that helped to shape this work in early stages, and to Universit\'e Aix-Marseille for hospitality during two visits.  

We are grateful to Vincent Guirardel for explaining some ideas from his new work with G. Levitt \cite{GL10}; this lead to an improved exposition in Sections \ref{S.Measures} and \ref{S.Decompose} coming from the use of the iterative procedure explained in Remark \ref{R.GLProcedure}.  More importantly, Guirardel pointed out a gap in an earlier approach: instead of Proposition \ref{P.ExceptionalMeasure}, we were using a result of Plante \cite[Theorem 3.1]{Pl75}, which Guirardel noticed to be incorrect.  Guirardel gave an enlightening counterexample and pointed out that \cite[Theorem 3.2]{Pl75} can be proved under a stronger hypothesis (mixing).  After appropriately weaking the hypothesis of Guirardel to suit our needs, we arrived at the proof of Proposition \ref{P.ExceptionalMeasure}, which generalizes a result of \cite{GL10}.

Very special thanks go to my advisor Ilya Kapovich for helpful guidiance, patient support, and for devoting so much of his time to my graduate education.

The author acknowledges support from National Science Foundation grant DMS 08-38434 ”EMSW21-MCTP: Research Experience for Graduate Students.

\section{Background}\label{S.Back}

In this section we briefly review the relevant definitions around $\mathbb{R}$-trees, Outer space, and algebraic laminations.  In what follows $F_n$ denotes the free group of rank $n$; for $g \in F_n$ let $[g]$ denote the conjugacy class of $g$.

\subsection{Basics About $\mathbb{R}$-Trees}

A metric space $(T,d)$ is called an $\mathbb{R}$-\emph{tree} (or just a \emph{tree}) if for any two points $x,y \in T$, there is a unique topological arc $p_{x,y}:[0,1] \rightarrow T$ connecting $x$ to $y$, and the image of $p_{x,y}$ is isometric to the segment $[0,d(x,y)]$.  As is usual, we let $[x,y]$ stand for Im$(p_{x,y})$, and we call $[x,y]$ the \emph{segment} (also called an \emph{arc}) in $T$ from $x$ to $y$.  A segment is called \emph{non-degenerate} if it contains more than one point.  We let $\overline{T}$ stand for the metric completion of $T$.  Unless otherwise stated, we regard $T$ as a topological space with the metric topology.  If $T$ is a tree, and $x \in T$, then $x$ is called a \emph{branch point} if the cardinality of $\pi_0(T - \{x\})$ is greater than two.  For $x \in T$, the elements of $\pi_0(T- \{x\})$ are called \emph{directions} at $x$.

In this paper, all the trees we consider are equipped with an isometric (left) action of a finitely generated group $G$, i.e. a group morphism $\rho: G \rightarrow$ Isom$(T)$; as usual, we always supress the morphism $\rho$ and identify $G$ with $\rho(G)$.  A tree $T$ equipped with an isometric action will be called an $G$-\emph{tree}, and we denote this situation by $G \curvearrowright T$.  Notice that an action $G \curvearrowright T$ induces an action of $G$ on the set of directions at branch points of $T$.  We identify two $G$-trees $T,T'$ if there is an $G$-equivariant isometry between them.

There are two sorts of isometries of trees: an isometry $g$ of $T$ is called \emph{elliptic} if $g$ fixes some point of $T$, while an isometry $h$ of $T$ is called \emph{hyperbolic} if it is not elliptic.  It is easy to see that any hyperbolic isometry $h$ of $T$ leaves invariant a unique isometric copy of $\mathbb{R}$ in $T$, which is called the \emph{axis} of $h$ and denoted by $A(h)$.  If $g$ is an elliptic isometry, we let $A(g)$ stand for the fixed point set of $g$, \emph{i.e.} $A(h):=\{x \in T|hx=x\}$.  Given a $G$-tree $T$, we have the so-called \emph{hyperbolic length function} $l_T:G \rightarrow \mathbb{R}$, where
$$l_T(g):=\inf \{d(x,gx)|x \in T\}$$ 

\noindent The number $l_T(g)$ is called the \emph{translation length} of $g$, and it is easily verified that, for any $g \in F_N$, the infimum is always realized on $A(g)$, so that $g$ acts on $A(g)$ as a translation of length $l_T(g)$.  If $H \leq G$ is a finitely generated subgroup containing a hyperbolic isometry, then $H$ leaves invariant the set
$$T_H:=\cup_{l_T(h)>0} A(h)$$

\noindent which is a subtree of $T$, and is minimal in the set of $H$-invariant subtrees of $T$; $T_H$ is called the \emph{minimal invariant subtree for} $H$.  An action $G \curvearrowright T$ is called \emph{minimal} if $T=T_G$; a minimal action $G \curvearrowright T$ is \emph{non-trivial} if $T$ contains more than one point.

For an action $G \curvearrowright T$, and for $x \in T$, let $Gx:=\{gx|g \in G\}$ denote the \emph{orbit} of $x$.  An action $G \curvearrowright T$ has \emph{dense orbits} if for some $x \in T$, we have $\overline{Gx}=T$.  Note that if some orbit is dense, then every orbit is dense.  

\subsection{Outer Space and its Closure}\label{S.CVn}

Recall that an action $F_n \curvearrowright T$ is \emph{free} if for any $1 \neq g \in F_n$ one has $l_T(g) > 0$.  If $X \subseteq T$, then the \emph{stabilizer} of $X$ is $Stab(X):=\{g \in F_n|gX=X\}$--the setwise stabilizer of $X$.  We say that an action $F_n \curvearrowright T$ is \emph{very small} if:

\begin{enumerate}
 \item [(i)] $F_n \curvearrowright T$ is minimal,
 \item [(ii)] for any non-degenerate arc $I \subseteq T$, $Stab(I) = \{1\}$ or $Stab(I)$ is a maximal cyclic subgroup of $F_n$,
 \item [(iii)] stabilizers of tripods are trivial.
\end{enumerate}
 
An action $F_n \curvearrowright T$ is called \emph{discrete} (or \emph{simplicial}) if the $F_n$-orbit of any point of $T$ is a discrete subset of $T$; in this case $T$ is obtained by equivariantly assigning a metric to the edges of a (genuine) simplicial tree.  It is important to note that the metric topology is weaker than the simplicial topology if the tree is not locally compact.

Let $T,T'$ be trees; a map $f:T \rightarrow T'$ is called a \emph{homothety} if $f$ is $F_n$-equivariant and bijective, and if there is some positive real number $\lambda$ such that for any $x,y \in T$, we have $d_{T'}(f(x),f(y))=\lambda d_T(x,y)$; in this case $T,T'$ are called \emph{projectively equivalent} or \emph{homothetic}.  

\begin{defn}\label{D.CVn}
\noindent
\begin{enumerate}
 \item The \emph{unprojectivised Outer space} of rank $n$, denoted $cv_n$, is the topological space whose underlying set consists free, minimal, discrete, isometric actions of $F_n$ on $\mathbb{R}$-trees; it is equipped with the \emph{length function topology}.
 \item \cite{CV86} The \emph{Culler-Vogtmann} \emph{Outer space} of rank $n$, denoted $CV_n$, is the topological space whose underlying set consists of homothety classes of free, minimal, discrete, isometric actions of $F_n$ on $\mathbb{R}$-trees; it is equipped with the \emph{projective length function topology}.
 \item The \emph{unprojectivised closed Outer space} of rank $n$, denoted $\overline{cv}_n$, is the topological space whose underlying set consists of very small isometric actions of $F_n$ on $\mathbb{R}$-trees; it is equipped with the \emph{length function topology}.
 \item The \emph{closed Outer space} of rank $n$, denoted $\overline{CV}_n$, is the topological space whose underlying set consists of homothety classes of very small isometric actions of $F_n$ on $\mathbb{R}$-trees; it is equipped with the \emph{projective length function topology}.
\end{enumerate}

\end{defn}

It is known that a minimal $F_n$-tree is completely determined by its hyperbolic length function \cite{CM87}; so points in $CV_n$ can be thought of as projective classes of such length functions, \emph{i.e.} $CV_n \subseteq \mathbb{P}\mathbb{R}^{F_n}$; and $CV_n$ is topologized via the quotient of the weak topology on length functions.  It is the case that the closure $\overline{CV}_n$ of $CV_n$ is compact and consists precisely of homothety classes of very small $F_N$-actions on $\mathbb{R}$-trees \cite{CL95, BF94}.  For more background on $CV_n$ and its closure, see \cite{Vog02} and the references therein.

\subsection{Algebraic Laminations}\label{SS.Laminations}

Here, we present a brief and restricted view of dual laminations of $F_n$-trees; see \cite{CHL08a} and \cite{CHL08b} for a careful development of the general theory.  Let $\partial F_n$ denote the Gromov boundary of $F_n$--\emph{i.e.} the Gromov boundary of any Cayley graph of $F_n$; let $\partial^2(F_n):=\partial F_n \times \partial F_n - \Delta$, where $\Delta$ is the diagonal.  The left action of $F_n$ on a Cayley graph induces actions by homeomorphisms of $F_n$ on $\partial F_n$ and $\partial^2 F_n$.  Let $i: \partial^2 F_n \rightarrow \partial^2 F_n$ denote the involution that exchanges the factors.  An \emph{algebraic lamination} is a non-empty, closed, $F_n$-invariant, $i$-invariant subset $\mathscr{L} \subseteq \partial^2 F_n$.  

Fix an action $F_n \curvearrowright T$ with dense orbits; following \cite{LL03} (see also \cite{CHL08b}), we associate an algebraic lamination $L^2(T)$ to the action $F_n \curvearrowright T$. Let $T_0 \in cv_n$ (\emph{i.e.} the action $F_n \curvearrowright T_0$ is free and discrete), and let $f:T_0 \rightarrow T$ be an $F_n$-equivariant map, isometric when restricted to edges of $T_0$.  Say that $f$ has \emph{bounded backtracking} if there is $C > 0$ such that $f([x,y]) \subseteq N_C([f(x),f(y)])$, where $N_C$ denotes the $C$-neighborhood.  For $T_0 \in cv_n$, denote by $vol(T_0):=vol(T_0/F_n)$ the sum of lengths of edges of the finite metric graph $T_0/F_n$.

\begin{prop}\label{P.BBT}\cite[Lemma 2.1]{LL03}
Let $T \in \overline{cv}_n$; let $T_0 \in cv_n$; and let $f:T_0 \rightarrow T$ be equivariant and isometric on edges.  Then $f$ has bounded backtracking with $C=vol(T_0)$.  
\end{prop}
 
For $T_0 \in cv_n$, we have an identification $\partial T_0 \cong \partial F_n$.  If $\rho$ is a ray in $T_0$ representing $X \in \partial F_n$, we say that $X$ is $T$-\emph{bounded} if $f \circ \rho$ has bounded image in $T$; this does not depend on the choice of $T_0$ (see \cite{BFH97}).  

\begin{prop}\label{P.DefQ}\cite[Proposition 3.1]{LL03}
Let $T \in \overline{cv}_n$ have dense orbits, and suppose that $X \in \partial F_n$ is $T$-bounded.  There there is a unique point $Q(X) \in \overline{T}$ such that for any $f:T_0 \rightarrow T$, equivariant and isometric on edges, and any ray $\rho$ in $T_0$ representing $X$, the point $Q(X)$ belongs to the closure of the image of $f \circ \rho$ in $\overline{T}$.  Further the image of $f \circ \rho$ is a bounded subset of $\overline{T}$.  
\end{prop}

The (partially-defined) map $Q$ given above is clearly $F_n$-equivariant; in fact, it extends to an equivariant map $Q: \partial F_n \rightarrow \overline{T} \cup \partial T$, which is surjective (see \cite{LL03}).  The crucial property for us is that $Q$ can be used to associate to $T$ an algebraic lamination.  

\begin{prop}\label{P.DefL2}\cite{CHL08b}
Let $T \in \overline{cv}_n$ have dense orbits.  The set $L^2_Q(T):=\{(X,Y) \in \partial^2(F_n)|Q(X)=Q(Y)\}$ is an algebraic lamination.
\end{prop}

Following \cite{CHL08b}, we mention that there is different, perhaps more intuitive, procedure for defining $L^2(T)$.  Let $T \in \overline{cv}_n$ (not necessarily with dense orbits, but not free and discrete), and let $\Omega_{\epsilon}(T):=\{g \in F_n|l_T(g) < \epsilon\}$, where $l_T$ is the hyperbolic length function for the action $F_n \curvearrowright T$.  The set $\Omega_{\epsilon}(T)$ generates an algebraic lamination $L^2_{\epsilon}(T)$, which is the smallest algebraic lamination containing $(g^{-\infinity},g^{\infinity})=(...g^{-1}g^{-1},gg...) \in \partial^2(F_n)$ for every $g \in \Omega_{\epsilon}$.  One then defines $L^2_{\Omega}(T):= \cap_{\epsilon > 0} L^2_{\epsilon}(T)$.  In \cite{CHL08b} it is shown that for an action $F_n \curvearrowright T \in \overline{cv}_n$ with dense orbits, $L^2_{\Omega}(T)=L^2_Q(T)$, as defined above.  

\begin{defn}\label{D.L2}
Let $F_n \curvearrowright T \in \overline{cv}_n$ be an action with dense orbits.  The \emph{dual lamination of} $F_n \curvearrowright T$ is $L^2(T):=L^2_Q(T)=L^2_{\Omega}(T)$.  
\end{defn}

\subsection{The Observers Topology}\label{SS.Observe}

In \cite{CHL07}, a weaker topology on $\mathbb{R}$-trees is considered.  Let $T$ be a tree with Gromov boundary $\partial T$ and metric completion $\overline{T}$; put $\hat{T}:=\overline{T} \cup \partial T$.  The metric topology on $T$ canonically extends to $\overline{T}$, and we may extend this topology to $\hat{T}$ as follows: for a ray $\rho$ in $T$ representing $[\rho] \in \partial T$, a neighborhood basis at $[\rho]$ is taken to be the set of components of $T \setminus \{pt.\}$ meeting $\rho$ in a non-compact set.  For $p,q \in \hat{T}$, the \emph{direction} of $q$ at $p$ the component $d_p(q)$ of $\hat{T} \setminus \{p\}$ containing $q$.  The \emph{observers topology} $\hat{T}$ is the topology with subbbasis the collection of directions in $\hat{T}$.  

With this topology the map $Q:\partial F_n \rightarrow \overline{T} \cup \partial T$ is continuous \cite[Proposition 2.3]{CHL07}.  Further, when restricted to finite subtrees of $T$, the observers topology agrees with the metric topology; in particular, we have the following:

\begin{lemma}\cite{CHL07, CHL08b}\label{L.CompactFibers}
 Let $T \in \overline{cv}_n$ have dense orbits.  For any $x \in T$, the set $Q^{-1}(x) \subseteq \partial F_n$ is compact.
\end{lemma}

The aim of \cite{CHL07} is to investigate to what extent $L^2(T)$ determines $T$ for trees $T \in \overline{cv}_n$ with dense orbits.  For the following equip trees $T \in \overline{cv}_n$ with the metric topology, and equip $\hat{T}$ with the observers topology.

\begin{prop}\label{P.L2DeterminesT}\cite[Theorem I]{CHL07}
 Let $T_1,T_2 \in \overline{cv}_n$ have dense orbits.  Then $L^2(T_1)=L^2(T_2)$ if and only if $\hat{T_1}$ is homeomorphic to $\hat{T_2}$.  
\end{prop}

Let $T \in \overline{cv}_n$ have dense orbits; fix $q \in T$; and let $(p_k)$ be a sequence in $T$.  Put $I_m:= \cap_{k \geq m} [q,p_k]$, so $I_m=[q,r_m]$, and we have $I_m \subseteq I_{m+1}$.  The \emph{inferior limit} of $(p_k)$ from $q$ is the limit $\lim_q p_k:= \lim r_m$.  The following gives a characterization of convergence in $\hat{T}$:

\begin{lemma}\cite[Lemma 1.12]{CHL07}
 If a sequence $(p_k)$ in $\hat{T}$ converges to $p$, then for any $q \in \hat{T}$, $p=\lim_q p_k$.  
\end{lemma}

A map $f:T \rightarrow T'$ between trees $T,T'$ is \emph{continuous on segments} if for any finite segment $I \subseteq T$, the restriction $f|_I:I \rightarrow T'$ is continuous.  The following result, along with the approach of Sections \ref{S.Measures} and \ref{S.Decompose} provide a sort of converse of the work in \cite{CHL09}.

\begin{lemma}\label{L.Observers}
 Let $T, T' \in \overline{cv}_n$ have dense orbits, and suppose that there is an equivariant bijection $f:T \rightarrow T'$ that is continuous on segments.  Then $f$ extends to a unique homeomorphism $\hat{f}:\hat{T} \rightarrow \hat{T'}$; in particular, $L^2(T)=L^2(T')$.  
\end{lemma}

\begin{proof}
 Let $T, T'$, and $f$ as in the statement, and let $T_{obs},T_{obs}'$ denote $T,T'$ regarded as subspaces of $\hat{T},\hat{T'}$.  We first show that $f$ induces a homeomorphism $T_{obs} \rightarrow T_{obs}'$.  Let $p,q \in T$, and notice that $d_p(q)=\cup_{p \notin [q,q']} [q,q']$.  As $f$ is continuous on segments and bijective, we have that $f(d_p(q))=\cup_{f(p) \notin [f(q),f(q')]} [f(q),f(q')] = d_{f(p)}(f(q))$, hence $f$ is open.  Applying essentially the same argument to $f^{-1}$ gives that $f$ is continuous, hence $f$ is a homeomorphism $T_{obs} \rightarrow T_{obs}'$.  

 Let $p_k \in \hat{T}$ with $p_k \rightarrow p \in \hat{T} \setminus T$.  By the discussion following Proposition \ref{P.L2DeterminesT}, we have for any $q \in \hat{T}$, $p=\lim_q p_k$.  Set $I_m=\cap_{k \geq m} [q,p_k]=[q,r_m]$, so that $\lim r_m = p$.  Since $f$ is continuous on segments and bijective, we have that $f(I_m)=[f(q),f(r_m)]$, hence the sequence $f(r_m)$ has a well-defined limit $r' \in \hat{T'}$.  If $r' \in T'\subseteq \hat{T}'$, then there is $r'' \in T$ such that $f(r'')=r'$; in this case $f([q,r''])=[f(q),r']$.  Further $[q,r'']$ evidently contains each $I_m$, hence $[q,r'']$ contains $\overline{\cup_m I_m} \ni \lim r_m=p$, a contradiction.  Hence, $r' \in \hat{T'} \setminus T'$, and we define $\hat{f}(p)=r'$.  

 Now note that for any $p' \in \hat{T'} \setminus T'$ and any $q' \in T'$ there is a sequence $r_m'$ in $T'$ such that $[q,r_m'] \subseteq [q,r_{m+1}']$ with $p' = \lim r_m'$.  We find $q,r_m \in T$ such that $f(q)=q'$ and $f(r_m)=r_m'$, and it follows from the preceding arguments that there is a unique $p \in \hat{T} \setminus T$ with $\hat{f}(p)=p'$; hence $\hat{f}$ is bijective.  Futher, it is easy to check as above that $\hat{f}$ is continuous and open, so $\hat{f}$ is a homeomorphism.  The fact that $L^2(T)=L^2(T')$ then follows from Proposition \ref{P.L2DeterminesT}.
\end{proof}

\section{Irreducible Endomorphisms}\label{S.Irred}

Let $\phi: F_n \rightarrow F_n$ be an endomorphism.  The \emph{outer class of} $\phi$ is the set $\Phi:=\{\iota_g \circ \phi | g \in F_n\}$, where $\iota_g$ is the inner automorphism $\iota_g(f)=gfg^{-1}$; we call $\Phi$ an \emph{outer endomorphism}.  We will frequently be discussing both outer classes of endomorphisms and particular endomorphisms in a class; we will always use capital letters to denote outer classes and lower case letters to denote particular endomorphisms, surpressing futher comment when confusion is unlikely to arise.  Further, we will frequently take liberties in replacing $\phi$ by a power with little or no comment, as throughout we are studying asymptotic behavior.

\subsection{Topological Representatives}\label{SS.TopReps}

This subsection closely follows \cite{BH92}, to which the reader should refer for details.  The (n-petal) \emph{rose} is $R_n:=\vee_{i=1}^n S^1$, the wedge of $n$ copies of $S^1$; once and for all, we make the identification $F_n=\pi_1(R_n)$.  A \emph{marked graph} is a finite graph $G$ of rank $n$, along with a homotopy equivalence $\tau:R_n \rightarrow G$; this gives an action of $F_n$ on $\tilde{G}$ by deck transformations, hence an identification of $\pi_1(G)$ with $F_n$.  This action is well-defined up to conjugation, \emph{i.e.} up to choosing a preferred lift of a base point in $G$.  Denote by $V=V(G)=\{v_1, ... , v_l\}$ and $E=E(G)=\{e_1, ... , e_k\}$ the sets of vertices and edges of $G$, respectively.  

Let $\Phi$ be an outer endomorphism of $F_n$, and let $G$ be a marked graph.  A map $f: G \rightarrow G$ is a \emph{topological representative} for $\Phi$ if:

\begin{itemize}
 \item $f(V) \subseteq V$,
 \item for any $e \in E$, $f|_e$ is either locally injective, or $f(e)$ is a vertex, and
 \item $f$ induces $\Phi$.
\end{itemize}

Topological representatives always exist; one can take the obvious map with $G=R_n$.  Fix a topological representative $f: G \rightarrow G$.  The \emph{transition matrix} of $f$ is the $k \times k$ matrix $M(f)$ whose $(i,j)$-entry is the number of times the $f$-image of $e_j$ crosses $e_i$ (in either direction).  Any transition matrix is a non-negative integral matrix, and it is evident that $M(f)^r=M(f^r)$.  We say that a non-negative integral matrix $M$ is (fully) \emph{irreducible} if:

\begin{itemize}
 \item for any $(i,j)$, there is $N(i,j)$ such that $(M^{N(i,j)})_{i,j} > 0$, and 
 \item the prior condition holds for any power of $M$.
\end{itemize}

A subgraph $G_0 \subseteq G$ is called \emph{invariant} if $f(G_0) \subseteq G_0$.  The topological representative $f$ is \emph{admissible} if there is no invariant non-degenerate forest.  

\begin{defn}\label{D.Irred}\cite{BH92}
 An endomorphism $\phi:F_n \rightarrow F_n$ is \emph{irreduicble} if any admissible topological representative for $\Phi$ has irreducible transition matrix.
\end{defn}

A \emph{free factor system} $\mathscr{F}$ for $F_n$ is a decomposition $F_n=F_{n_1} \ast ... \ast F_{n_r} \ast F'$, where each $F_{n_i}$ is a non-trivial proper subgroup.  An endomorphism $\psi:F_n \rightarrow F_n$ \emph{preserves} $\mathscr{F}$ if $\psi(F_{n_i}) \leq F_{n_i}^{g_i}$ for some elements $g_i \in F_n$.  

\begin{lemma}\label{L.PreserveFF}\cite{BH92}
 If $\phi:F_n \rightarrow F_n$ be irreducible, then $\phi$ does not preserve any free factor system for $F_n$.
\end{lemma}

\begin{cor}
 Let $\phi: F_n \rightarrow F_n$ be an irreducible endomorphism.  Then $\phi$ is injective.
\end{cor}

\begin{proof}
 An easy argument using Nielsen moves shows that $ker(\phi)$ contains a non-trivial free factor of $F_n$; the corollary then follows from Lemma \ref{L.PreserveFF}.
\end{proof}

Let $f:G \rightarrow G$ be a topological representative with $M(f)$ irreducible.  A \emph{turn} in $G$ is a set $T=\{e_i,e_j\}$ of directed edges of $G$ with a common initial vertex; a turn is \emph{degenerate} if $e_i=e_j$.  The topological representative $f$ induces a map $Tf$ on the set of turns in $G$ by sending an edge $e$ to the first edge in the path $f(e)$.  A non-degenerate turn $T$ is called \emph{illegal} if $(Tf)^r(T)$ is degenerate for some $k$.  A turn is called \emph{legal} if it is not illegal, and a path is called legal if it crosses only legal turns.  For any path $\alpha$ in $G$, denote by $[f(\alpha)]$ the immersed path homotopic to $f(\alpha)$ (rel endpoints).  

\begin{defn}\label{D.TT}\cite{BH92}
 An admissible topological representative $f:G \rightarrow G$ for $\Phi$ is a \emph{train track map} for $\Phi$ if $[f^r(e)]=f^r(e)$ for every $e \in E$.  
\end{defn}

\subsection{Train Tracks and the Stable Tree}\label{SS.TrainTracks}

The following result is proved in \cite{BH92} for irreducible outer automorphisms of $F_n$; however, with no modification their proof works for all irreducible endomorphisms.  This result is established by different means by Dicks-Ventura in \cite{DV}.

\begin{prop}\label{P.TT}\cite{BH92, DV}
 Let $\phi: F_n \rightarrow F_n$ be irreducible, then $\Phi$ has a topological representative that is a train track map.
\end{prop}

The Perron-Frobenius theory gives for any any irreducible matrix $M$ a unique positive normalized eigenvector $\mathbf{v}$ with associated eigenvalue $\lambda > 1$ (see \cite{Sen}).  Let $\phi: F_n \rightarrow F_n$ be irreducible, and let $f: G \rightarrow G$ be a train track representative for $\phi$.  We equip $G$ with the \emph{Perron-Frobenius metric}: identify edge $e_i$ with the segment of length $\mathbf{v}_i$.  With this metric, the map $f$ expands lengths of legal paths by the factor $\lambda$.  For $g \in F_n$ we let $\alpha_g$ stand for the immersed loop representing $[g]$.  

\begin{lemma}\label{L.StableT}\cite{BFH97}
 Let $\phi:F_n \rightarrow F_n$ be irreducible, and let $f:G \rightarrow G$ be a train track representative for $\Phi$ with Perron-Frobenius eigenvalue $\lambda$.  For $g \in F_n$ put
$$
l_{T_\Phi}(g):= \lim_k \lambda^{-k} L([f^k(\alpha_g)])
$$
Then the following hold:

\begin{enumerate}
 \item [(i)] $l_{T_\Phi}$ is the length function for an $\mathbb{R}$-tree $T_\Phi \in \overline{cv}_n$,
 \item [(ii)] $T_{\Phi}$ is independent of the choice of train track representative, and
 \item [(iii)] $l_{T_\Phi}(\phi(g))=\lambda_{\Phi} l_{T_\Phi}(g)$,
\end{enumerate}

\end{lemma}

Let $\phi:F_n \rightarrow F_n$ be irreducible, and let $f:G \rightarrow G$ be a train track representative for $\Phi$ with associated Perron-Frobenius eigenvalue $\lambda$.  Note that it follows from that above lemma that $\lambda$ is determined by $\Phi$.  Equip $G$ with the Perron-Frobenius metric, and put $T_0:=\tilde{G}$ with the lifted metric.  Let $\tilde{f}:T_0 \rightarrow T_0$ be a lift of $f$; the choice of $\tilde{f}$ amounts to choosing some representative $\psi \in \Phi$, and we prefer to take $\tilde{f}$ corresponding to $\phi$ when convenient.  Note that for any $g \in F_n$, one has $\tilde{f}(gx)=\phi(g)\tilde{f}(x)$.  Let $T_k'$ denote the minimal $F_n$-invariant subtree of $T_0$ with the action twisted by $\phi^k$; so $T_k'=\tilde{f}^k(T_0)$.  Define $T_k$ to be $T_k'$ with the metric rescaled by $\lambda^{-k}$.  The sequence of $F_n$-trees $(T_k)$ converges in the Gromov-Hausdorff topology to the tree $T_{\Phi}$.  The map $\tilde{f}:T_0 \rightarrow T_0$ gives maps $f_k:T_k \rightarrow T_{k+1}$, which give rise to a map $f_{\phi}:T_{\Phi} \rightarrow T_{\Phi}$ satisfying:

\begin{itemize}
 \item Length$(f_{\phi}([x,y]))=\lambda_{\Phi}$Length$([x,y])$,
 \item $f_{\phi}(gx)=\phi(g)f_{\phi}(x)$
\end{itemize}

\begin{defn}
 The tree $T_{\Phi}$ is called the \emph{stable tree of} $\Phi$.
\end{defn}

Any endomorphism $\psi:F_n \rightarrow F_n$ acts (on the right) on the set of $F_n$-trees via
$$
l_{T\psi}(g)=l_T(\psi(g))
$$
\noindent If one restricts attention to a space $X$ of nontrivial trees such that $\phi$ acts on $X$, then the action of $\phi$ on $X$ gives an action on the set of projective classes of trees coming from $X$.  If $[T]\psi=[T]$, then $T\psi$ is $F_n$-equivariantly isometric to $T$ with the metric rescaled by some number $c$.  This data is witnessed by a function $H:T \rightarrow T$ satisfying:

\begin{itemize}
 \item Length$(H([x,y]))=c$Length$([x,y])$,
 \item $H(gx)=\psi(g)H(x)$
\end{itemize}

Call such a map $H$ a $\psi$-\emph{compatible} $c$-\emph{homothety}, or just a \emph{homothety} if $\psi$ and $c$ are clear from context (see \cite{GJLL98}).  Conversely, if $Y$ is an $F_n$ tree, $\eta:F_n \rightarrow F_n$ some endomorphism, then the existence of a $\eta$-compatible homothety $H:Y \rightarrow Y$ implies that $[Y]\eta=[Y]$.  The map $f_{\phi}:T_{\Phi} \rightarrow T_{\Phi}$ is a $\phi$-compatible $\lambda_{\Phi}$-homothety, so $[T_{\Phi}]\phi=[T_{\Phi}]$.

\subsection{Expansive Endomorphisms}

\begin{defn}\label{D.Expansive}
 Fix a basis $B$ for $F_n$.  An endomorphism $\phi: F_n \rightarrow F_n$ is \emph{expansive} with respect to $B$ if for any real number $L$ there is a number $K$ such that for any $1 \neq g \in F_n$, one has $||\phi^k(g)||_B \geq L$ whenever $k \geq K$.  
\end{defn}

The definition of expansive involves a particular basis for $F_n$; however, it is clear that if an endomorphism is expansive with respect to some basis then it is expansive with respect to any basis.  The following lemma is an easy application of the definition of a train track map.

\begin{lemma}\label{L.1}\cite{BFH97}
 Let $f:G \rightarrow G$ be a train track map with associated eigenvalue $\lambda$, and let $p$ be a path in $G$.  Then the sequence $L([f^k(p)])$ either is uniformly bounded or grows like $Const. \lambda^k$.  
\end{lemma}

\begin{cor}\label{C.TisFree}
 Let $\phi:F_n \rightarrow F_n$ be irreducible, and suppose that $\phi$ is expansive.  Then $T_{\Phi}$ is free.
\end{cor}

\begin{proof}
 Considering Lemma \ref{L.1} and the construction of $T_{\Phi}$ we see that an elements $g \in F_n$ is elliptic in $T_{\Phi}$ if and only if the conjugacy class of $g$ is represented by a loop $\alpha_g$ in $G$ such that the length of $[f^k(\alpha_g)]$ is uniformly bounded.  If $\phi$ is expansive, this is only possible for $g=1$.
\end{proof}

We now establish a dichotomy for irreducible endomorphisms of the free group.  The result follows easily from a theorem of M. Takahasi \cite{T51}(see also \cite{KM02}), but we include a proof, as the techniques are relevant to the sequel.


\begin{prop}\label{P.Dichot}
 Let $\phi:F_n \rightarrow F_n$ be irreducible, then either $\phi \in Aut(F_n)$ or $\phi$ is expansive.
\end{prop}

\begin{proof}
 Suppose that $\phi: F_n \rightarrow F_n$ is irreducible and not expansive.  Let $S_k=S(\phi^k(F_n))$ denote the Stallings subgroup graph of $\phi^k(F_n)$, and let $x_k \in S_k$ be the base point (see \cite{KM02} for background).  Denote by $i_k$ the injectivity radius of $S_k$ (with the simplicial metric).  Since $\phi$ is not expansive, it follows that the sequence $(i_k)_{k \in \mathbb{N}}$ is bounded.  Hence, by replacing $\phi$ by a suitable power, we can find a labeled graph $S'$ with basepoint $x' \in S'$ along with embeddings $f_k:S' \rightarrow S_k$ sending $x'$ to the projection of $x_k$ onto the image of $f_k$.  

 Let $f:G \rightarrow G$ be a train track representative for $\Phi$.  Since the action $F_n \curvearrowright \tilde{R_n}$ is quasi-isometric to the action $F_n \curvearrowright \tilde{G}$, after replacing $\phi$ by a power if necessary, we get that the collection of subgraphs of $S_k$ that are unions of short loops gives rise to a collection of subgroups that invariant under $\phi$ up to conjugacy.  Each of these subgroups is a free factor of $F_n$ as its conjugacy class corresponds to a subgraph of each $S_k$ for $k>>0$, which implies that there is a free factorization of $F_n$ mapping onto this collection.  Since $\phi$ is injective and irreducible, we get that $S'$ surjects onto each $S_k$ so that $\phi$ is an automorphism.
\end{proof}

\begin{remark}
The above proof shows that for any endomorphism $\psi:F_n \rightarrow F_n$ that is not expansive, after possibly replacing $\psi$ by a power, we can find a collection of free factors $F_{n_1},...,F_{n_r}$ of $F_n$ that are preserved by $\psi$ up to conjugacy such that the restriction of $\psi$ to each $F_{n_i}$ is an automorphism.  In this case, there are inner automorphisms $\iota_{g_{n_i}}$ such that $\cap_k (\iota_{g_{n_i}} \circ \psi)^k(F_n)=F_{n_i}$.  
\end{remark}

\section{Graphs of Actions}\label{S.GoA}

To begin in earnest our study of the structure of stable trees of irreducible endomorphisms, we recall the following notion of decomposability for group actions on trees.  Let $G$ be finitely generated group acting on an $\mathbb{R}$-tree $T$.  

\begin{defn}\label{D.TF}
 A $G$-invariant family $\mathscr{Y}=\{Y_v\}_{v \in V}$ of non-degenerate subtrees $Y_v \subseteq T$ is called a \emph{transverse family} for the action $G \curvearrowright T$ if for $Y_v \neq Y_{v'}$, one has that $Y_v \cap Y_{v'}$ contains at most one point.     
\end{defn}

Note that if $\mathscr{Y}$ is a transverse family for the action $G \curvearrowright T$, we may replace each $Y_v$ by its closure in $T$; the resulting collection also will be a transverse family.  Let $\{Y_v\}_{v \in V}$ is a transverse family of closed subtrees of $T$.  If, in addition, for any finite arc $I \subseteq T$, one has that $I$ is contained in a finite union $Y_{v_1} \cup ... \cup Y_{v_r}$, then the collection $\mathscr{Y}$ is called a \emph{transverse covering} of $T$ \cite{Gui08}.

\begin{defn}\label{D.GoA}\cite{Lev94, Gui08}
A \emph{graph of actions} $\mathscr{G}=(S, \{Y_v\}_{v \in V(S)}, \{p_e\}_{e \in E(S)}$ consists of:

\begin{enumerate}
 \item [(i)] a simplicial tree $S$, called the \emph{skeleton}, equipped with an action (without inversions) of $G$,
 \item [(ii)] for each vertex $v \in V(S)$ of $S$ an $\mathbb{R}$-tree $Y_v$, called a \emph{vertex tree}, and
 \item [(iii)] for each oriented edge $e \in E(S)$ with terminal vertex $v \in V(S)$ a point $p_e \in Y_v$, called an \emph{attaching point}.
\end{enumerate}

\end{defn}

It is required that the projection sending $Y_v \rightarrow p_e$ is equivariant and that for $g \in G$, one has $gp_e=p_{ge}$.  Associated to a graph of actions $\mathscr{G}$ is a canonical action of $G$ on an $\mathbb{R}$-tree $T_{\mathscr{G}}$: define a pseudo-metric $d$ on $\coprod_{v\in V(S)} Y_v$: if $x \in Y_u$, $y \in Y_v$, let $e_1...e_k$ be the reduced edge-path from $u$ to $v$ in $S$, \emph{i.e.} $\iota(e_1)=u$, $\tau(e_k)=v$, and $\tau(e_i)=\iota(e_{i+1})$, then 

\begin{equation}\label{E.Dist}
d(x,y)=d_{Y_u}(x,p_{\overline{e_1}}) + d_{Y_{\tau(e_1)}}(p_{e_1},p_{\overline{e_2}}) + ... +d_{Y_v}(p_{e_r},y)
\end{equation}

\noindent Making this pseudo-metric Hausdorff gives an $\mathbb{R}$-tree, called the \emph{dual} of $\mathscr{G}$, which we denote by $T_{\mathscr{G}}$.  If $T$ is an $\mathbb{R}$-tree equipped with an action of $G$ by isometries, and if there is an equivariant isometry $T \rightarrow T_{\mathscr{G}}$ to the dual of a graph of actions, then we say that $T$ \emph{splits} as a graph of actions.  See \cite{Gui04, Lev94} for details.  

The following result shows that graphs of actions and transverse coverings are equivalent ideas.

\begin{lemma}\label{L.Graphs}\cite[Lemma 1.5]{Gui08}
 Assume that $T$ splits as a graph of actions with vertex trees $\{Y_v\}_{v \in V(S)}$, then $\{Y_v\}_{v \in V(S)}$ is a transverse covering for $T$.  Conversely, if $T$ has a transverse covering $\{Y_v\}_{v \in V}$, then $T$ splits as a graph of actions whose non-degenerate vertex trees are $\{Y_v\}_{v \in V}$.
\end{lemma}

We now recall a sketch of the proof of the second statement of Lemma \ref{L.Graphs}.  Suppose that the action $G \curvearrowright T$ has transverse covering $\mathscr{Y}=\{Y_v\}_{v \in V}$; we find a graph of actions structure for $G \curvearrowright T$.  First we define the skeleton $S$; $V(S)=V_0 \cup V_1$, where the elements of $V_0$ are in one-to-one correspondence with the elements of $\mathscr{Y}$, and the elements of $V_1$ are in one-to-one correspondence with intersection points of distinct elements of $\mathscr{Y}$.  There is an edge from $v_1 \in V_1$ to $v_0 \in V_0$ if and only if the point corresponding to $v_1$ is contained in the tree corresponding to $v_0$.  One checks that there is an induced action of $G$ on $S$ without inversions and that association given above of trees to the elements of $V(S)$ defines a graph of actions structure on $G \curvearrowright T$ (see \cite{Gui08}).  

The following is a simple application of the preceding discussion.

\begin{lemma}
 Let $G \curvearrowright T$ be an action of a finitely generated group on an $\mathbb{R}$-tree, and suppose that $\mathscr{T}=\{T_v\}_{v \in V}$ is a transverse covering for $G \curvearrowright T$.   If the action $G \curvearrowright T$ is free, then each $Stab(T_v)$ is a free factor of $G$.
\end{lemma}

\begin{proof}
 Let $\mathscr{G}=(S, \{T_v\}_{v \in V(S)}, \{p_e\}_{e \in E(S)})$ be the graph of actions structure on $G \curvearrowright T$ defined above.  Note that edge stabilizers in the action $G \curvearrowright~S$ arise from stabilizers of attaching points.  Since $G \curvearrowright T$ is free, edge stabilizers in $G \curvearrowright S$ are trivial.  Since vertex stabilizers in $G \curvearrowright S$ correspond to the stabilizers of the trees $Y_v$, we see from the Bass-Serre theory that each such stabilizer is a free factor of $G$.  
\end{proof}

To conclude this section, we state the following result of Levitt.

\begin{prop}\label{C.IndiscreteGraph}\cite[Theorem 5]{Lev94}
 Let $G \curvearrowright T$ be an action of a finitely generated group on an $\mathbb{R}$-tree; suppose that the action is not simplicial and not with dense orbits.  Then $G \curvearrowright T$ splits as a graph of actions $\mathscr{G}=(S, \{T_v\}_{v \in V(S)}, \{p_e\}_{e \in E(S)}$, where each $T_v$ is either a finite segment or $Stab(T_v) \curvearrowright T_v$ is with dense orbits.
\end{prop}

If $T \in \overline{cv}_n$ does not have dense orbits, then there is some discrete orbit; according to the above result the union of discrete orbits in $T$ is a forest $F$ with a positive lower bound on the diameter of each component.  The set of components of $T \setminus F$ consists of finitely many orbits of subtrees of $T$, such that the stabilizer of each component acts on it with dense orbits.  The union of components of $T \setminus F$ with closures of components of $F$ is a transverse covering of the action $F_n \curvearrowright T$.  We call the set of closures of components of $F$ the \emph{simplicial part of} $T$.

\begin{remark}
 Graphs of actions are ubiquitous in the sequel, so it seems appropriate to give a bit of motivation; for this we reach to the source of the idea.  The definition of a graph of actions generalizes the decomposition of a tree dual to surface lamination that comes from the decomposition of the lamination into its minimal components.  Indeed, if a surface $S$ is equipped with measured lamination $(L,\mu)$, and if $T$ is the $\mathbb{R}$-tree dual to $(L,\mu)$, then there is, for each sublamination $L'$ of $L$, a transverse family $\mathscr{T}_{L'}$ of subtrees in $T$ that are dual to the various lifts of $L'$ to $\tilde{S}$.  It is easy to see that if $L_1,...,L_k$ are the minimal sublaminations of $L$, then there is a transverse covering, namely $\mathscr{T}=\mathscr{T}_{L_1} \cup ... \cup \mathscr{T}_{L_k}$, of $T$, containg $k$ orbits of trees; this corresponds to the decomposition of $L$ into minimal components.
\end{remark}

\section{Structure of the Stable Tree}\label{S.StableTree}

In this section we investigate the structure of the stable tree of an irreducible non-surjective endomorphism of $F_n$; the first step is to show that if some orbit is discrete, then every orbit must be discrete.  To that end we recall the following result of Levitt-Lustig:

\begin{lemma}\label{L.ShortBasis}\cite[Corollary 2.5]{LL03}
 Let $T \in \overline{cv}_n$ have dense orbits.  Given $p \in T$ and $\epsilon>0$, there is a basis $\{a_1,...,a_n\}$ of $F_n$ such that $\Sigma_{i=1}^n d(p,a_ip) < \epsilon$.
\end{lemma}

The next lemma is our first step in characterizing the structure of the stable tree of an irreducible, non-surjective endomorphism of $F_n$.

\begin{lemma}\label{L.NoGraph}
 Let $\phi:F_n \rightarrow F_n$ be an irreducible endomorphism, and let $T_{\Phi}$ be its stable tree.  Then either the action $F_n \curvearrowright T_{\Phi}$ has dense orbits or $F_n \curvearrowright T_{\Phi}$ is free and simplicial.
\end{lemma}

\begin{proof}
 In the case that $\phi \in Aut(F_n)$ this result follows from \cite{BFH97}.  Hence, in light of Proposition \ref{P.Dichot}, we may assume that $\phi$ is expansive, and by Corollary \ref{C.TisFree} we have that $T_{\Phi}$ is free.  Toward a contradiction we assume that the action $F_n \curvearrowright T_{\Phi}$ is not discrete but does not have dense orbits.  In this case Corollary \ref{C.IndiscreteGraph} gives that $F_n \curvearrowright T_{\Phi}$ splits as a graph of actions with vertex trees simplicial edges or trees with dense orbits.  

 Put $T=T_{\Phi}$, and let $f:T \rightarrow T$ be a homothety witnessing $[T]\phi=[T]$.  Immediately one has that for $\epsilon>0$, $\phi$ takes elements of $\epsilon$-short translation length to elements of $\lambda \epsilon$-short translation length.  Recall that the action $F_n \curvearrowright T\phi$ is precisely the action $\phi(F_n) \curvearrowright T_{\phi(F_n)}$ of $\phi(F_n)$ on its minimal invariant subtree $T_{\phi(F_n)}$.  There are finitely many orbits of vertices in the skeleton of the graph of actions structure on $T$; each vertex group either acts with dense orbits on the corresponding vertex tree or is trivial, in the case that the corresponding vertex tree is contained in the simplicial part of $T$.  

 Let $\mathscr{G}=(S,\{T_v\}_{v \in V(S)}, \{p_e\}_{e \in E(S)})$ be the graph of actions structure on $T$ guaranteed by Proposition \ref{C.IndiscreteGraph} and described above.  As the action $F_n \curvearrowright T$ is free, the action $F_n \curvearrowright S$ is a free decomposition of $F_n$.  Choose representatives $V_1,...,V_r$ of conjugacy classes of vertex groups with $V_i=Stab(T_{v_i})$ such that the action $V_i \curvearrowright T_{v_i}$ has dense orbits.  According to Lemma \ref{L.ShortBasis}, for any $\epsilon > 0$ and points $p_i \in T_{v_i}$, we can find bases $B_i$ for $V_i$ such that $\Sigma_{b \in B_i} d(p_i,bp_i) < \epsilon$.  Taking $\epsilon$ small with respect to the minimal length of a simplicial edge in $T$ and recalling Formula \ref{E.Dist}, we see that each $B_i$ is mapped under $\phi$ into a single vertex group of the graph of actions structure on $F_n \curvearrowright T$.  Since there finitely many conjugacy classes of these vertex groups, it follows that there is some $V_j$ such that $\phi(V_j) \leq V_j^g$ for some $g \in F_n$; by Lemma \ref{L.PreserveFF}, we arrive at a contradiction to irreducibility of $\phi$.
\end{proof}

Let $T$ be a tree with base point $x \in T$.  To each $x \neq y \in T$, there is associated a (one-sided) \emph{cylinder} $C_x(y)$ that consists of rays $\rho$ in $T$ based at $x$ that contain the segment $[x,y]$.  The cylinder $C_x(y)$ is regarded as a subset of $\partial T$. 

To complete this section, we bring the following result, ruling out the possiblity that the stable tree $T$ of a non-surjective irreducible endomorphism $\phi$ could have dense orbits.  The proof of this result contains a characterization of the dynamics of $\partial \phi$ acting on $\partial F_n$ (Corollary \ref{C.BoundaryDynamics}), which we see to be incompatible with the existence of a map $Q$ for $T$ (refer to Subsection \ref{SS.Laminations}).

\begin{prop}\label{P.FreeSimp}
 Let $\phi:F_n \rightarrow F_n$ be an irreducible endomorphism, and assume that $\phi \notin Aut(F_n)$.  Then $T_{\Phi}$ is free and simplicial.
\end{prop}

\begin{proof}
 Put $T=T_{\Phi}$, $f=f_{\phi}:T \rightarrow T$, $\lambda=\lambda_{\Phi}$, and note that since $\phi$ is not an automorphism, we have $\lambda > 1$ and that $\phi$ is expansive.  Further, by Corollary \ref{C.TisFree}, $T_{\Phi}$ is free, and by Lemma \ref{L.NoGraph} $T_{\Phi}$ either has dense orbits, or $T_{\Phi}$ is simplicial.  

 Toward a contradiction suppose that the action $F_n \curvearrowright T$ is free with dense orbits.  By \cite{GL95} there are finitely many $F_n$-orbits of branch points in $T$ and finitely many orbits of directions at branch points in $T$.  By the equation $f(gx)=\phi(g)f(x)$ we have that $f$ induces a well-defined map on the set of orbits of branch points in $T$.  By replacing $f$ with some power, we get a branch point $x \in T$ such that $f(x)=gx$, for some $g \in F_n$.  Replace $f$ by $g^{-1}f$, which is easily seen to be a homothety representing $\iota_{g^-1} \circ \phi$.  This gives $f(x)=x$.  As the map $f$ is a homothety, it is injective; since there are finitely many directions at $x$, we may replace $f$ by a power to ensure that $f$ fixes each direction at $x$.  

 Let $d$ be some direction at $x$, and let $\rho$ be a ray in $T$ based at $x$ in direction $d$.  It follows that there is $y \in d$ such that $[x,y] \subseteq f(\rho) \cap \rho$.  Since $f$ is a $\lambda$-homothety and since $\lambda > 1$, we can find a sequence $y_k \in d$ such that $f^k([x,y_k])=[x,y]$.  It follows that $[x,y] \subseteq \cap_k f^k(T)$.

 Let $Q=Q_T: \partial F_n \rightarrow T$ be the map defined in Proposition \ref{P.DefQ}.  Recall that $Q$ is $F_n$-equivariant and surjective, so for any $z \in [x,y]$ the set $Q^{-1}(z) \subseteq \partial F_n$ is non-empty, and by Lemma \ref{L.CompactFibers} $Q^{-1}(z)$ is compact.  The commutativity of the below diagram follows easily from the definition of the map $Q$; see Subsection \ref{SS.Laminations}.
 \[
      \xymatrix{
      \partial F_n \ar[r]^{\partial \phi} \ar[d]_Q & \partial F_n \ar[d]^Q \\
      T \ar[r]^h & T
      }
 \]
 By definition $(\partial \phi)^k(\partial F_n)=\partial \phi^k(F_n)$.  As $[x,y] \subseteq \cap_k f^k(T)$, for each $z \in [x,y]$, we have that the sets $Z_k:=Q^{-1}(z) \cap (\partial \phi)^k(\partial F_n)$ form a nested sequence of non-empty compact subsets of $\partial F_n$, so $Z:= \cap_k Z_k$ is non-empty.  Hence, $\partial \phi(\cup_{z \in [x,y]}Z) = \cup_{z \in [x,y]}Z \subseteq \cap_k \partial(\phi^k(F_n))$; in particular, $\cap_k \partial(\phi^k(F_n))$ is infinite.  We show that this is impossible.

 Fix a Cayley tree $T'$ for $F_n$.  Let $S_k':=S(\phi^k(F_n))$ be the Stallings subgroup graph for $\phi^k(F_n)$, and let $S_k:=Core(S_k')$ be the core of $S_k'$ (see \cite{KM02}).  A fundamental domain for the action $\phi^k(F_n) \curvearrowright T_{\phi^k(F_n)}'$ can be got by ``unfolding'' $S_k$ in $T_{\phi^k(F_n)}$, and such a fundamental domain is the union of exactly $2n$ (possibly overlapping) segments eminating from $1 \in T'$.  It follows that $\partial \phi^k(F_n)$ is contained in the union of $2n$ cylinders, say $C_{1,k},...,C_{2n,k}$.  Let $g_{i,k} \in F_n$ be chosen to define $C_{i,k}$.  Notice that since $\phi$ is expansive we have for any $N$ some $k(N)$ such that $l_{T'}(g_{i,k(N)}) \geq N$ for each $i$.  It follows that $\cap_k(\cup_i C_{i,k})$ is a finite set; on the other hand $\cap_k \partial(\phi^k(F_n)) \subseteq \cap_k(\cup_i C_{i,k})$, a contradiction.
\end{proof}

As the dynamics of $\partial \phi$ acting on $\partial F_n$ is of independent interest, we include the following corollary, which follows immediately from the above proof.

\begin{cor}\label{C.BoundaryDynamics}
 Let $\phi:F_n \rightarrow F_n$ be an irreducible endomorphism, and assume that $\phi$ is not an automorphism.  The induced map $\partial \phi:\partial F_n \rightarrow F_n$ has finitely many fixed points $X_1,...,X_r$ such that $r\leq 2n$, and each $X_i$ is attracting.  If $N \subseteq \partial F_n$ is some compact neighborhood of $\{X_1,...,X_r\}$ then there is $K$ such that $(\partial \phi)^k(\partial F_n) \subseteq N$ for any $k \geq K$.
\end{cor}

The following corollary is a restatement of Proposition \ref{P.FreeSimp} in the language of train tracks.

\begin{cor}\label{C.NoIllegalTurns}
 Let $\phi:F_n \rightarrow F_n$ be an irreducible endomorphism, and assume that $\phi$ is not an automorphism.  Then $\Phi$ is topologically represented by a train track map with no illegal turns.  
\end{cor}

\begin{proof}
 From Proposition \ref{P.FreeSimp} we have that the action $F_n \curvearrowright T_{\Phi}$ is free and simplicial.  Let $f:T_{\Phi} \rightarrow T_{\Phi}$ be a homothety witnessing the fact that $[T_{\Phi}]\phi=[T_{\Phi}]$; then $f$ descends to a map $\overline{f}:T_{\Phi}/F_n \rightarrow T_{\Phi}/F_n$ that is easily seen to be a simplicial immersion inducing $\Phi$, \emph{i.e.} a train track representative with no illegal turns.  
\end{proof}

\section{Dynamics on $CV_n$}\label{S.CVnDynamics}

In this section we classify the dynamics of an irreducible non-surjective endomorphism $\phi$ acting on $CV_n$.  Recall that, in this case, by Proposition \ref{P.Dichot}, we have that $\phi$ is expansive; and by Proposition \ref{P.FreeSimp} there is a fixed point for the action, namely $[T_{\Phi}]$.

\subsection{The Stable Lamination}\label{SS.StableLam}

Let $\phi:F_n \rightarrow F_n$ be irreducible.  Following \cite{BFH97} we associate to $\Phi$ an algebraic lamination.  Let $f:G \rightarrow G$ be a train track representative for $\Phi$ with transition matrix $M=M(f)$, and equip $G$ with the Perron-Frobenius metric (see Subsection \ref{SS.TrainTracks}).  By Corollary \ref{C.NoIllegalTurns}, we can assume that $f$ is an immersion.  

Let $e_i \in E(G)$; by irreducibility of $M$, there is a natural number $k$ such that the $(i,i)$-entry of $M^k$ is non-zero.  Since $M^k=M(f^k)$, this gives that the $f$-image of $e_i$ crosses $e_i$.  This gives a fixed point $x$ of $f^k$ in the interior of $e_i$.  Let $N(x)$ be a small $\epsilon$-neighborhood of $x$ in the interior of $e_i$.  There is a unique orientation-preserving isometry $l_0:(-\epsilon, \epsilon) \rightarrow N(x):0 \mapsto x$.  Each $f^r|_{e_i}$ is an immersion; hence, there are unique orientation-preserving isometric immersions $l_n:(-\lambda^{nk} \epsilon, \lambda^{nk} \epsilon) \rightarrow G: 0 \mapsto x$ satisfying $l_n(y)=f^k(l_{n-1}(\lambda^{-k}y))$.  The sequence $(l_n)$ gives an isometric immersion $l:\mathbb{R} \rightarrow G$ that is $f^k$-invariant in the sense that $f^k \circ l:\mathbb{R} \rightarrow G$ is a reparametrization of $l$.

Let $L_{\Phi}$ stand for the set of isometric immersions $l':\mathbb{R} \rightarrow G$ obtained via the above procedure; this set is essentially the lamination defined in \cite{BFH97}.  The marking $\tau:R_n \rightarrow G$ gives a free action of $F_n=\pi_1(R_n)$ on $\tilde{G}$, which gives an identification $\partial F_n \cong \partial G$.  This gives a homeomorphism from the space of immersed lines in $\tilde{G}$ (with the weak topology) to $\partial^2 F_n$.  For any $l \in L_{\Phi}$ there are various lifts of $l$ to $\tilde{G}$, and the collection of lifts to $\tilde{G}$ of lines $l \in L_{\Phi}$ evidently gives an $F_n$-invaraiant subset $\mathscr{L}_{\Phi} \subseteq \partial^2 F_n$.  The \emph{stable lamination of} $\Phi$, denoted $\Lambda_{\Phi}$, is defined to be the smallest algebraic lamination containing $\mathscr{L}_{\Phi}$.   

\subsection{The Convergence Criterion}

In this subsection we state a result of Bestvina-Feighn-Handel from \cite{BFH97} that gives a sufficient condition on a tree $T \in \overline{cv}_n$ to ensure that $[T]\phi^k$ converges to $[T_{\Phi}]$; this will immediately give a dynamics statement for an irreducible, non-surjective endomorphism acting on Outer space.  

Let $T_0 \in cv_n$ and $T \in \overline{cv}_n$; an equivariant map $f:T_0 \rightarrow T$ has \emph{bounded backtracking} if there is a constant $C$ such that the $f$-image of a segment $[p,q]$ is contained in the $C$-neighborhood of the segment $[f(p),f(q)]$.  The smallest such $C$ is called the \emph{backtracking constant of} $f$, denoted $BBT(f)$.  It is a fact that for $T_0, T$, and $f$ as above, it is always the case that $f$ has bounded backtracking (see \cite{LL03} and the references therein).

\begin{prop}\label{P.ConvergeLemma}\cite{BFH97}\cite[Proposition 6.1]{LL03}
 Let $T \in \overline{cv}_n$.  Suppose that there is a tree $T_0 \in cv_n$, an equivariant map $f:T_0 \rightarrow T$, and a bi-infinite geodesic $\gamma_0 \subseteq T_0$ representing a leaf of $\Lambda_{\Phi}$ such that $f(\gamma_0)$ has diameter greater than $2BBT(f)$.  Then $f(\gamma_0)$ has infinite diameter and there exists a neighborhood $V$ of $[T]$ in $\overline{CV}_n$ such that $\phi^p|_V$ converges uniformly to $[T_{\Phi}]$.
\end{prop}

We cite the result of \cite{LL03}, as it is completely clear that their proof works in our context; Proposition 6.1 is proved for laminations associated to irreducible outer automorphisms of $F_n$, but the proof goes through without modification for the case of non-surjective irreducible endomorphisms.  Actually, the proof could be simplified by considering only the case of an irreducible expansive endomorphism, as one has in this case the luxury of a train track representative with no illegal turn.

Proposition \ref{P.ConvergeLemma} can be restated in terms of dual laminations.  Let $T \in \overline{cv}_n$ have dense orbits.  For any $\epsilon >0$, Proposition 2.2 of \cite{LL03} ensures the existence of a simplicial tree $T_0 \in cv_n$ and an equivariant map $f:T_0 \rightarrow T$ with $BBT(f) < \epsilon$.  If $Z=(X,Y) \in \partial^2 F_n$ is some point such that for all $f$ with small backtracking, a line representing $Z$ in $T_0$ is sent under $f$ to a small diameter subset of $T$, then $Z \in L^2(T)$.  Hence, we may apply Proposition \ref{P.ConvergeLemma} to get convergence for a tree $T \in \overline{cv}_n$ as long as some leaf of $\Lambda_{\Phi}$ is not contained in $L^2(T)$; note that by irreducibility of $\phi$, if some leaf of $\Lambda_{\Phi}$ is contained in $L^2(T)$, then every leaf of $\Lambda_{\Phi}$ is contained in $L^2(T)$.  

\begin{cor}\label{C.InteriorDynamics}
 Let $\phi:F_n \rightarrow F_n$ be irreducible and non-surjective.  For any $[T] \in CV_n$, we have $[T]\phi^k \rightarrow [T_{\Phi}]$. 
\end{cor}

\begin{proof}
 For any $T \in cv_n$ and any $Z \in \partial^2 F_n$, $Z$ is represented by an infinite line in $T$; the result follows by applying Proposition \ref{P.ConvergeLemma}.
\end{proof}

The convergence of Corollary \ref{C.InteriorDynamics} is uniform on compact subsets of $CV_n$; the goal of the next several sections is to show that the convergence is actually uniform over all of $CV_n$.  The next section deals with obvious obstructions.

\section{Endomorphisms Acting on $\overline{CV}_n$}\label{S.Admissible}

It is evident that $\phi$ acts on $CV_n$ as long as $\phi$ is injective.  However, for a tree $T \in \overline{cv}_n$, it could be the case that $T\phi$ is trivial even if $\phi$ is injective, and in this case $\phi$ would not act on $\overline{CV}_n$.  The aim of this section is to first illustrate exactly how an endomorphism can fail to act on $\overline{CV}_n$ and to give a sufficient condition for an irreducible endomorphism to act on $\overline{CV}_n$.

\subsection{Admissible Endomorphisms}

\begin{example}\label{E.NonAdmissible}
 Let $F_3=F(a,b,c)$, and define $\phi:F_3 \rightarrow F_3$ by:
\begin{align*}
 a &\mapsto a\\
 b &\mapsto bab^{-1}\\
 c &\mapsto b^2ab^{-2}
\end{align*}
 Suppose that $T \in \overline{cv}_3$ is the Bass-Serre tree of the splitting $F(a,b,c)=\langle a,b \rangle \ast \langle c \rangle$.  The endomorphism $\phi$ is injective, but the tree $T\phi$ is trivial, since $\phi(F_3)$ fixes the vertex of $T$ corresponding to $\langle a,b \rangle$.
\end{example}  

The endomorphism $\phi$ does not act on $\overline{CV}_3$, hence we must restrict attention to a proper subsemigroup of the semigroup of injective endomorphisms of $F_n$.

\begin{defn}\label{D.Admissible}
 An endomorphism $\phi:F_n \rightarrow F_n$ is called \emph{admissible} if for all $T \in \overline{cv}_n$, one has that $T\phi$ is non-trivial.
\end{defn}

It follows from \cite{Gui98} that any very small action $F_n \curvearrowright T$ with trivial arc stabilizers can be approximated by a simplicial very small action $F_n \curvearrowright T'$ such that a subgroup $V \leq F_n$ fixes a point $x \in T$ if and only if $V$ fixes some vertex $x' \in T'$; hence we get the following characterization of admissibility.

\begin{lemma}\label{L.Simp}
 An endomorphism $\phi: F_n \rightarrow F_n$ is admissible if and only if for any simplicial tree $T \in \overline{cv}_n$, one has that $T\phi$ is non-trivial.
\end{lemma}

Lemma \ref{L.Simp} shows that Example \ref{E.NonAdmissible} is quite generic and immediately emphasizes the importance of vertex stabilizers in simplicial trees in $\overline{cv}_n$ in the present context.  Hence, we bring the following:

\begin{defn}
 A splitting of $F_n$ is called \emph{very small} if it corresponds to a simplicial tree in $\overline{cv}_n$.  
\end{defn}

The following is a translation of \cite[Definition 2.2]{BFH97} into the formalism of algebraic laminations.

\begin{defn}
 Let $H \leq F_n$ be finitely generated; say that (the conjugacy class of) $H$ \emph{carries} a point $Z \in \partial^2 F_n$ if for any $T \in cv_n$, there is $g \in F_n$ such that $Z \in \partial^2 T_{H^g}$.  
\end{defn}

The stable lamination $\Lambda_{\Phi}$ and its relationship to the dual lamination $L^2(T)$ of a tree $T \in \overline{cv}_n$ is of primary importance to us if we wish to apply the convergence criterion given by Proposition \ref{P.ConvergeLemma}; hence, we now begin working to develop a characterization of admissibility involving only $\Lambda_{\Phi}$.

\begin{lemma}\label{L.FinIndex}
 Let $T \in cv_n$, and let $H \leq F_n$ be finitely generated.  Then $T_H=T$ if and only if $H$ is finite index in $F_n$.
\end{lemma}

\begin{proof}
 Let $T$ and $H$ be as in the statement, and suppose that $T_H =T$.  First suppose that $T=\tilde{R_n}$ is the ``standard'' Cayley tree for $F_n=F(A)$, and regard $T$ as a labeled directed tree.  Then $T/H$ is a labeled directed finite graph, which is $A$-regular (see \cite{KM02}), and choosing a basepoint in $T/H$ gives an immersion representing a subgroup $H' \leq F_n$ that is conjugate to $H$.  As $T/H$ is $A$-regular, this immersion is a covering map, and it follows that $H'$ is finite index in $F_n$; hence, $H$ is finite index in $F_n$.

 Now let $T \in cv_n$ be arbitrary, and choose a spanning tree $G_0 \subseteq T/F_n$; collapsing the lifts of $G_0$ in $T$ to points gives a map $f:T \rightarrow T_0$ onto a Cayley tree.  By replacing $H$ by its image under some $\alpha \in Aut(F_n)$, we may suppose that $T_0$ is the  Cayley tree $\tilde{R_n}$.  It is easy to see that $(T_0)_H=f(T_H)$, and so by above, we get that $H$ is finite index in $F_n$.

 Conversely, suppose that $H \leq F_n$ is finite index, so there is $k$ such that for all $f \in F_n$, one has $f^k \in H$.  Since $A(f^k)=A(f)$, it follows that $T_H=\cup_{1 \neq h \in H} A(h)=T$.  
\end{proof}

\subsection{The Admissibility Criterion}

We now establish a characterization of admissibility in the case of irreducible endomorphisms; the following lemma allows us to use the convergence criterion of Proposition \ref{P.ConvergeLemma} to understand the action of an admissible irreducible endomorphism on the simplicial trees in $\overline{CV}_n$.

\begin{prop}\label{P.LamLemma}
 Let $\phi:F_n \rightarrow F_n$ be irreducible.  Then $\phi$ is admissible if and only if no leaf of $\Lambda_{\Phi}$ is carried by a vertex group of a very small splitting of $F_n$.
\end{prop}

\begin{proof}
 If $\phi$ is not admissible, then, by Lemma \ref{L.Simp}, there is some simplicial tree $Y \in \overline{cv}_n$, such that $Y\phi$ is trivial.  In this case some vertex group of the splitting corresponding to $Y$ carries every leaf of $\Lambda_{\Phi}$.

 So, assume that $\phi$ is admissible.  Put $T:=T_{\Phi}$, and let $f:T \rightarrow T$ be the $\phi$-compatible $\lambda:=\lambda_{\Phi}$ homothety of $T$.  As in the proof of Proposition \ref{P.FreeSimp}, after possibly passing to a power of $f$, we may find a branch point $x \in T$ that is fixed by $f$ and such that every direction at $x$ is fixed by $f$ as well.  It is easy to see that there is, in each diretion at $x$, an infinite ray based at $x$ that is fixed setwise by $f$.  Denote this infinite multipod by $X$; it follows that $\partial^2 X \subseteq \Lambda_{\Phi}$.

 Toward a contradiction, suppose that some leaf of $\Lambda=\Lambda_{\Phi}$ is carried by a vertex group $V$ of a very small splitting of $F_n$; without loss, we can assume that $V$ is a vertex group of a one-edge splitting.  As $\phi$ is irreducible, we have that every leaf of $\Lambda$ is carred by $V$ (see \cite{BFH97}).  The following is easily verified:

 \begin{claim}\label{C.VertexIntersection}
  Let $H \leq F_n$ be finitely generated, and let $Y \in cv_n$.  There is a constant $C=C(H,Y)$ such that if $Y_H \cap Y_{H^g}$ has diameter greater than $C$, then $Y_H \cap Y_{H^g}$ has infinite diameter, and $H \cap H^g$ is non-trivial.  
 \end{claim}

 \noindent Note that if $V$ is a vertex group of a very small 1-edge splitting and if $V \neq V^g$ with $V \cap V^g$ nontrivial, then $V \cap V^g$ is cyclic and is conjugate to the edge group of the splitting.

 \begin{claim}\label{C.NonPeriodic}
  No leaf of $\Lambda$ is periodic: for any tree $T \in cv_n$ and any non-trivial $g \in F_n$, there is a constant $K=K(T,g)$ such that if $l$ is a line in $T$ representing a leaf of $\Lambda$, then the diameter of $l \cap A(g)$ is bounded above by $K$.
 \end{claim}
  \begin{proof}
   The existence of a periodic leaf would give an element $f \in F_n$ such that $\phi(f)=hf^rh^{-1}$ for some $r$ and some $h \in F_n$.  If $|r|=1$, expansivity of $\phi$ is contradicted. If $|r|>1$, it follows from \cite[Lemma 4.1]{BF94} that $f$ is primitive, contradicting Lemma \ref{L.PreserveFF}.
  \end{proof}
 
 According to the above claims, after possibly replacing $V$ with a conjugate, we have that $X \subseteq T_V$.  As $\phi$ is irreducible, any line in $X$ crosses every $\phi(F_n)$-orbit of branch points in $T_{\phi(F_n)}$ infinitely often.  Let $C=C(V,T)$ be the constant guaranteed by Claim \ref{C.VertexIntersection}.  By replacing $\phi$ by a power if necessary, we can assume that branch points in $T_{\phi(F_n)}=f(T)$ are separated by distance at least $C$.  Suppose that $c \in F_n$ is a generator for some (non-trivial) cyclic intersection $V \cap V^f$, and put $K=K(T,c)$ as in Claim \ref{C.NonPeriodic}; increase $C$ if necessary to ensure that $C \geq K$.

 Let $y \in T$ be a branch point of $T_{\phi(F_n)}$ such that $[x,y]$ contains no other branch points of $T_{\phi(F_n)}$.  By the above discussion, there is $g \in F_n$ such that $gT_V$ contains an infinite multipod centered at $y$ that also contains the point $x$.  Hence the diameter of $T_v \cap gT_v$ is infinite and $V \cap V^g$ is nontrivial.  If it were the case that $V \neq V^g$, then $V \cap V^g$ is cyclic; however, this is impossible by choice of $C$.  It follows that $V=V^g$, and by iterating this argument, we get that $T_{\phi(F_n)} \subseteq T_V$.  

 \begin{claim}
  Let $Stab(T_V)$ denote the setwise stabilizer, then $Stab(T_V)=V$. 
 \end{claim}

 \begin{proof}
  If $Stab(T_V) \neq V$, there is a finitely generated $V'$ containing $V$ such that $T_{V'}=T_V$.  It follows from Lemma \ref{L.FinIndex} that $V$ is finite index in $V'$.  Let $Y$ be a Bass-Serre tree for a very small splitting of $F_n$ with vertex group $V$.  It is easy to see that $V'$ fixes the vertex of $Y$ corresponding to $V$, hence $V'=V$.
 \end{proof}

 It follows from the above claim that for any finitely generated $K\leq F_n$, if $T_K \leq T_V$, then $K \leq V$.  Therefore, we conclude that for some $k$, $\phi^k(F_n) \leq V$, contradicting admissibility of $\phi$.
\end{proof}

\section{Convergence for Simplicial Actions and Graphs of Actions}\label{S.SimpGraphConverge}

In this section we apply Proposition \ref{P.LamLemma} along with Proposition \ref{P.ConvergeLemma} to understand the action of a non-surjective admissible irreducible endomorphism on tree that splits as a non-trivial graph of actions.

\subsection{Simplicial Actions in $\overline{CV}_n$}

\begin{prop}\label{P.SimpConverge}
 Let $\phi:F_n \rightarrow F_n$ be irreducible and non-surjective, and suppose that $\phi$ is admissible.  For any simplicial $T \in \overline{cv}_n$, one has $[T]\phi^k \rightarrow [T_{\Phi}]$.
\end{prop}

\begin{proof}
 By Proposition \ref{P.LamLemma} we have that no leaf of $\Lambda_{\Phi}$ is carried by a vertex group of a very small splitting of $F_n$.  By \cite[Proposition 1.3]{Bow98} $\partial F_n$ is naturally identified with the disjoint union of $\partial T$ with the union of boundaries of the vertex stabilizers.  It follows that for any leaf $Z \in \Lambda_{\Phi}$ and any equivariant map $f:T_0 \rightarrow T$ from a simplicial tree $T_0$ to $T$, if $l$ is a line in $T_0$ representing $Z$, then $f(l)$ has infinite diameter in $T$.  Convergence follows from Proposition \ref{P.ConvergeLemma}.
\end{proof}

\subsection{Graphs of Actions in $\overline{CV}_n$}

\begin{lemma}\label{L.EventuallyFree}
 Let $\phi:F_n \rightarrow F_n$ be irreducible, non-surjective, and admissible; let $T \in \overline{cv}_n$.  There is $k$ such that $T\phi^k$ is free.
\end{lemma}

\begin{proof}
 Toward a contradiction suppose that there is $T \in \overline{cv}_n$ such that for all $i$, $T\phi^i$ is not free.  By \cite{Gui98} point stabilizers in $T$ are vertex groups of a very small splitting of $F_n$; there are finitely many orbits of points in $T$ with nontrivial stabilizer, so finitely many conjugacy classes of such vertex groups appear.  Hence, we may find a sequences $g_k \in \phi^{i_k}(F_n)$ and $h_k \in F_n$, and a vertex group $V$ of a very small splitting of $F_n$ such that $g_k^{h_k} \in V$.  

 Again by \cite{Gui98}, there is a simplicial tree $T' \in \overline{cv}_n$ such that the set of point stabilizers in $T$ is equal to the set of vertex stabilizers in $T'$.  As $g_k^{h_k}$ fixes a point in $T\phi^{i_k}$, we have that $g_K^{h_k}$ fixes a point in $T'\phi^{i_k}$.  Since, $T_{\Phi}$ is free, we arrive at a contradiction to Proposition \ref{P.SimpConverge}.
\end{proof}

We now consider trees in $\overline{cv}_n$ that split as graphs of actions.  Let $T\in \overline{cv}_n$, and suppose that $T$ splits as a grpah of actions $T=T_{\mathscr{G}}$, for $\mathscr{G}=(S, \{T_v\}_{v \in V(S)}, \{p_e\}_{e \in E(S)})$ (refer to Section \ref{S.GoA} for definitions).  

\begin{lemma}\label{L.GraphsConverge}
 Let $\phi:F_n \rightarrow F_n$ be irreducible, non-surjective, and admissible; and let $T \in \overline{cv}_n$.  If $T$ splits as a graph of actions, then $[T]\phi^k \rightarrow [T_{\Phi}]$.  
\end{lemma}

\begin{proof}
 Let $\phi$ as in the statement, and suppose that $T$ splits as a graph of action $\mathscr{G}=(S, \{T_v\}_{v \in V(S)}, \{p_e\}_{e \in E(S)})$.  According to Lemma \ref{L.EventuallyFree}, there is $k$ such that the action $T\phi^k$ is free.  Keep in mind that the action $T\phi^k$ is equivariantly isometric to the action $\phi^k(F_n) \curvearrowright T_{\phi^k(F_n)}$.  Put $H_l:=\phi^l(F_n)$, and put $T_l:=T\phi^l$; we regard the action $T_l$ as a subaction of the action $F_n \curvearrowright T$; namely, the action $F_n \curvearrowright T_l$ is precisely the action $H_l \curvearrowright T_{H_l}$.  

 The subgroups $H_l$ act on $S$, and it is evident that the union of vertex trees in $T$ corresponding to vertices of $S_{H_l}$ give a transverse covering of $T_l$, whence $T_l$ inherits a graph of actions structure from $\mathscr{G}$, with skeleton $S_l:=S_{H_l}$.  As $T_k$ is free, the action $H_l \curvearrowright S_l$ is with trivial arc stabilizers; this is because arc stabilizers in the action $F_n \curvearrowright S$ correspond to stabilizers of attaching points.  Thus it follows from Proposition \ref{P.SimpConverge} that there is $M$ such that $H_l \curvearrowright S_l$ is free for $l \geq M$ and such that for any non-trivial $h \in H_l$, the translation length of $h$ is at least two in the simplicial metric on $S_l$.  Hence, for $l$ big enough, $S_l$ is locally finite, and it follows from the distance formula for graphs of actions (Formula \ref{E.Dist}) that there is a positive lower bound for translation lengths of non-trivial elements for the action $H_l \curvearrowright T_l$.  Hence, for $l$ big enough, $T_l$ is free and simplicial, and by Proposition \ref{P.SimpConverge} we have that $[T]\phi^k \rightarrow [T_{\Phi}]$.
\end{proof}

\section{Convergence for Indecomposable Actions in $\overline{CV}_n$}\label{S.IndecompConverge}

In this section we consider trees with the following strong mixing property introduced by Guirardel in \cite{Gui08}; this definition is crucial for the sequel.

\begin{defn}\label{D.Indecomposable}
 An action $G \curvearrowright T$ of a finitely generated group on an $\mathbb{R}$-tree is called \emph{indecomposable} if for any finite, non-degenerate arcs $I,J \subseteq T$, there are elements $g_1,...,g_r \in G$ such that $J \subseteq g_1I \cup ... \cup g_rI$ and such that $g_iI \cap g_{i+1}I$ is non-degenerate for $i \leq r-1$.
\end{defn}

It is important to note that the intersections $g_iI \cap J$ can be degenerate; see \cite{Gui08} for further discussion.  The following result of \cite{R10a} allows us to handle convergence for indecomposable trees in $\overline{cv}_n$.

\begin{prop}\label{P.IndecompConverge}\cite[Theorem 4.4]{R10a}
Let $T \in \overline{cv}_n$ be indecomposable, and let $H \leq F_n$ be a finitely generated subgroup.  The action $H \curvearrowright T_H$ has dense orbits if and only if $H$ is finite index in $F_n$.  
\end{prop}

Note that Proposition \ref{P.IndecompConverge} implies that for finitely generated, infinite index $H \leq F_n$, it must be the case that $H \curvearrowright T_H$ is simplicial.  Indeed, if not, by Proposition \ref{C.IndiscreteGraph} there would be some finitely generated $K \leq H$ such that $K \curvearrowright T_K$ has dense orbits.

\begin{cor}
 Let $\phi:F_n \rightarrow F_n$ be irreducible, non-surjective, and admissible, and let $T \in \overline{cv}_n$ be indecomposable.  Then $[T]\phi^k \rightarrow [T_{\Phi}]$.  
\end{cor}

\begin{proof}
 Let $\phi$ and $T$ as in the statement.  As $\phi$ is non-surjective, we have that $\phi(F_n)$ has infinite index in $F_n$; it then follows from Proposition \ref{P.IndecompConverge} and the discussion above that $\phi(F_n) \curvearrowright T_{\phi(F_n)}$ is simplicial.  Convergence then follows from Proposition \ref{P.SimpConverge}.
\end{proof}

\section{Invariant Measures and Projections}\label{S.Measures}

In this section we establish some structure theory for trees $T \in \overline{cv}_n$ that do not split as graphs of actions and are not indecomposable; in short, we show how to find $T' \in \overline{cv}_n$ such that $L^2(T) \subseteq L^2(T')$ and such that either $T'$ splits as a graph of actions, or $T'$ is indecomposable.  The aim, of course, is to obtain convergence for the remainder of trees in $\overline{cv}_n$.  The main technical tool is the notion of a \emph{length measure}; as mentioned in the Introduction, this tool treats a tree $T \in \overline{cv}_n$ as a generalization of a measured lamination on a surface: the length measures are analogs of the transverse measures.

\subsection{Length Measures}\label{SS.LM}

Let $T$ be an $\mathbb{R}$-tree.  The following definition appears in \cite{Gui00}, where it is attributed to F. Paulin.

\begin{defn}
 A \emph{length measure} (or just \emph{measure}) $\mu$ on $T$ is a collection $\mu=\{\mu_I\}_{I\subseteq T}$ of finite positive Borel measures on the finite arcs $I \subseteq T$; it is required that for $J \subseteq I$ $\mu_J=(\mu_I)|_J$. 
\end{defn}

As these measures are defined locally on finite arcs, all the usual measure-theoretic definitions are similarly defined: a set $X \subseteq T$ is $\mu$-\emph{measurable} if $X \cap I$ is $\mu_I$-measurable for each $I \subseteq T$; $X$ is $\mu$-\emph{measure zero} if $X \cap I$ is $\mu_I$-measure zero for each $I$; and so on.  The \emph{Lebesgue length measure} on $T$, denoted $\mu_L$, is the collection of Lebesgue measures on the finite arcs of $T$.  

If $T$ is equipped with an action of a group $G$, then we say that a (length) measure $\mu$ is $G$-\emph{invariant} if $\mu_I(X \cap I) = \mu_{g.I}(g.X \cap g.I)$ holds for each $g \in G$.  Note that if the action $G \curvearrowright T$ is by isometries, then the Lebesgue measure is invariant.  We let $M(T)=M(G \curvearrowright T)$ stand for the set of invariant measures on $T$.  The following lemma shows that the existence of an invariant atomic measure has a simple interpretation.

\begin{lemma}\label{L.SplitAtoms}
Suppose that $G \curvearrowright T$ has an invariant atomic measure, then $T$ splits as a graph of actions
\end{lemma}

\begin{proof}
Let $\mu$ be a $G$-invariant atomic measure on $T$; without loss, we suppose that $\mu$ is ergodic.  Let $x \in T$ with $\mu(x) >0$.  Since the measures $\mu_I$ are finite, it follows that $G.x$ meets any finite subtree of $T$ in a finite set.  Consider the collection $\{T_v\}_{v \in V}$ of closures of components of $T \setminus G.x$; this family is evidently a transverse covering of $T$.  Hence, $G \curvearrowright T$ splits as a graph of actions by Lemma \ref{L.Graphs}.
\end{proof}

Later we will restrict our attention to non-atomic measures; Lemma \ref{L.SplitAtoms} shows that this restriction is vacuous as long as the tree in question does not split as a graph of actions.  The following definition from \cite{Gui08} is convenient when dealing with length measures; the discussion following it shows that global properties of length measures can be seen in finite subtrees.

\begin{defn}
Let $G$ a group and $T$ an $\mathbb{R}$-tree equipped with an action of $G$ be isometries; and let $K \subseteq T$ be a subtree.  We say that the action $G \curvearrowright T$ is \emph{supported} on $K$ if for any finite arc $I \subseteq T$, there are $g_1, ...,g_r \in G$ such that $I \subseteq g_1K \cup ... \cup g_rK$.  
\end{defn}

Let $G$ and $T$ as above, and suppose that $G$ is finitely generated with generating set $X$.  Then for any $y \in T$ the convex hull of $\{gy\}_{g \in X^{\pm}}$ is a finite supporting subtree for the action $G \curvearrowright T$.  For a measure $\mu \in M_0(T)$ and a finite tree $K=I_1 \cup ... \cup I_l \subseteq T$ for finite arcs $I_j \subseteq T$, let $Supp_K(\mu)$ denote the union of support sets $Supp(\mu_{I_1}) \cup ... \cup Supp(\mu_{I_l})$.  The set $Supp_K(\mu)$ is called the \emph{$K$-support} of $\mu$; if $K$ is clear from context, then $Supp_K(\mu)=Supp(\mu)$ is called the \emph{support} of $\mu$.  If $X \subseteq T$ is some subset, say that the support of $\mu$ is contained in $X$, if for every finite $K \subseteq T$, one has $Supp_K(\mu) \subseteq K \cap X$; similarly write $Supp(\mu)=X$ if for every finite $K \subseteq T$, one has $Supp_K(\mu)=K \cap X$.   

Recall that given an action $G \curvearrowright T$, $M(T)$ denotes the positive convex cone of $G$-invariant measures on $T$.  A non-trivial measure $\mu \in M(T)$ is called \emph{ergodic} if any $G$-invariant subset is either full measure or measure zero; the $G$-tree is called \emph{uniquely ergodic} if there is a unique, up to scaling, $G$-invariant measure $\mu$ on $T$; in this case $\mu$ must be ergodic.  Let $M_0(T)$ denote the set of non-atomic, $G$-invariant measures on $T$, and let $M_1(T) :=\{\nu \in M_0(T)|\nu \leq \mu_L\}$.  Note that both $M_0(T)$ and $M_1(T)$ are convex.

We equip $M_0(T)$ and $M_1(T)$ with the \emph{weak topology}: a sequence $\mu_i$ converges to $\mu$ if for every finite arc $I$ and every continuous functional $f$ on $I$, we have $\int_I f d(\mu_i)_I \rightarrow \int_I f d\mu_I$.  If $G \curvearrowright T$ is an action with finite supporting subtree $K \subseteq T$, then $M(T)$ can be identified with the space of (ordinary) maeasures $\nu$ on $K$ that are invariant under the (closed) pseudogroup $\Gamma=\Gamma_K$ generated by the restrictions $g|_K:g^{-1}K \cap K \rightarrow K \cap gK$  It should be noted that $\Gamma$ differs from a pseudogroup in the usual sense in that the domains of elements of $\Gamma$ are closed.  It is important to keep in mind the following issue: choose an ennumeration $G=\{g_1,g_2,...\}$, and suppose that there is a sequence $\mu_l$ of probability measures on $K$ with $\mu_l$ invariant under the restrictions $\{g_1|_K,...,g_l|_K\}$.  Since the domains of $g_i|_K$ are closed, it does not follow that $\mu=\lim \mu_l$ is invariant under $\{g_1|_K, g_2|_K,...\}$; see \cite{Gui00} for further discussion.  

The following result shows that certain actions in the closure of Outer space are finite dimensional from the current measure-theoretic point of view.

\begin{prop}\label{P.M0T}\cite[Corollary 5.4]{Gui00}
Let $T \in \overline{cv}_n$ be with dense orbits. Then $M_0(T)$ is a finite dimensional convex set, which is projectively compact. Moreover, $T$ has at most $3n-4$ non-atomic ergodic measures (up to homothety), and every measure in $M_0(T)$ is a sum of these ergodic measures.
\end{prop}

Let $T \in \overline{cv}_n$, and suppose that $T$ has dense orbits; then Proposition \ref{P.M0T} enusres that there is a finite set $\{\nu_1,...,\nu_k\} \subseteq M_0(T)$ of mutually singular ergodic measures spanning $M_0(T)$.  The following simple proposition shows that the supports of these ergodic measures are arranged in $T$ in a simple way; the result follows from the definition of ergodicity and the fact that the ``topological dyanmics'' of $F_n \curvearrowright T$ determine the way $F_n$-invaraint sets can be arranged in $T$.  

\begin{prop}\label{P.SuppProperties}
With notation as above:
\begin{enumerate}
 \item [(i)] if $I\subseteq Supp(\nu_i)$ is non-degenerate, then $\nu_i(I) > 0$,
 \item [(ii)] $K = \cup Supp(\nu_i)$,
 \item [(iii)] if $Supp(\nu_i) \cap Supp(\nu_j)$ contains a set of positive $\nu_i$-measure, then $Supp(\nu_i) \subseteq Supp(\nu_j)$,
 \item [(iv)] if $Supp(\nu_i) \cap Supp(\nu_j)$ contains a non-degenerate arc, then $Supp(\nu_i) = Supp(\nu_j)$.  
\end{enumerate}
\end{prop}

\begin{proof}
The statement (i) is immediate from the definition.  For (ii), we have that $\cup Supp(\nu_i)$ is a $\mu_L$ full measure subset of $K$, hence dense; but it is closed, so $\cup Supp(\nu_i) = K$.  For (iii) note that, by ergodicity of $\nu_i$, the union of $\Gamma$-translates of $Supp(\nu_i) \cap Supp(\nu_j)$ is a $\nu_i$ full measure subset of $Supp(\nu_i)$, hence dense in $Supp(\nu_i)$.  On the other hand, the union of $\Gamma$-transaltes of $Supp(\nu_i) \cap Supp(\nu_j) \subseteq Supp(\nu_j)$ as $Supp(\nu_j)$ is $\Gamma$-invariant.  Since $Supp(\nu_j)$ is closed, it follows that $Supp(\nu_i) \subseteq Supp(\nu_j)$.  The claim (iv) follows from (i) and (iii).  
\end{proof}

This immediately gives:

\begin{cor}\label{C.SuppProperties}
With notation as above, $K =\cup Supp(\nu_{i_j})$, where $\nu_{i_j}$ runs over measures in $M_1(T)$ whose supports contain a non-degenerate arc.
\end{cor}

\subsection{Pullbacks and Projections}

\begin{defn}\label{D.AlignPreserve}\cite{Gui00}
 Let $T$ and $T'$ be $\mathbb{R}$-trees.  A map $f:T \rightarrow T'$ is \emph{alignment-preserving} if for any $x \in T'$, $f^{-1}(x)$ is a convex subset of $T$.
\end{defn}

Suppose now that $T$ and $T'$ are $\mathbb{R}$-trees, equipped with actions by isometries of a finitely generated group $G$; and suppose that $f:T \rightarrow T'$ is $G$-equivariant and alignment-preserving.  It is observed in \cite{Gui00} that any non-atomic $\mu' \in M(T')$ can be pulled-back to a measure $\mu=f_*(\mu') \in M(T)$.  Indeed, we let $\mu$ be the unique non-atomic measure such that for all $J \subseteq I \subseteq T$, $\mu_I(J)=\mu'_{f(I)}(f(J))$.  As pointed out in \cite{Gui00}, the hypothesis on $f$ can be weakened; the construction goes through as long as $f$ is equivariant and any finite arc $I \subseteq T$ can be subdivided into finitely-many subarcs such that the restriction of $f$ to each subarc is alignment-preserving.  

Let $G \curvearrowright T$ be an action by isometries.  Given a $G$-invariant non-atomic measure $\mu$ on $T$, one may consider a pseudo-metric $d_{\mu}$ on $T$ defined by $d_{\mu}(x,y):=\mu([x,y])$.  It is easy to see that making this pseudo-metric Hausdorff gives an $\mathbb{R}$-tree $T_{\mu}$, equipped with an isometric action of $G$.  In the situation above, the natural map $f_{\mu}:T \rightarrow T_{\mu}$ is alignment-preserving, and if $\mu \leq \mu_L$, then $f_{\mu}$ is 1-Lipschitz.  

\begin{defn}
 If there is an equivariant, alignment-preserving map $f:T \rightarrow T'$, then we say that $T'$ is a \emph{projection} of $T$.
\end{defn}

Note that if $G \curvearrowright T$ is indiscrete but not with dense orbits, then the map $T \rightarrow T'$ collapsing each component of the simplicial part of $T$ to a point is a projection.  Indeed, in this case, by Proposition \ref{C.IndiscreteGraph} splits as a graph of actions with vertex trees either simplicial or with dense orbtis; it is easy to see that restricting the Lebesgue measure of $T$ to the trees with dense orbits gives and invariant measure $\mu \in M_0(T)$ such that $T'=T_{\mu}$.

\subsection{Exceptional Sets}

Let $T \in \overline{cv}_n$ have dense orbits.  

\begin{defn}\label{D.Exceptional}
 An invariant subset $X \subseteq T$ is called \emph{exceptional} if for any finite subtree $K \subseteq T$,  $X \cap K$ is closed and nowhere dense and if there is a finite subtree $K_0 \subseteq T$ such that $X \cap K_0$ contains a Cantor set.
\end{defn}  

A famous theorem of Imanishi \cite{Ima79}, rediscovered by Morgan-Shalen \cite{MS84} and proved in the present context by Levitt \cite{GLP94}, states that given a finite 2-complex $A$, equipped with a codimension-1 singular measured foliation, one is able to cut the 2-complex $A$ along certain subsets of singular leaves to arrange that every leaf is either finite or locally dense (see \cite{GLP94} or \cite{BF95}); the key property is that no leaf closure is a Cantor set.  

An action $G \curvearrowright T$ of a finitely preseted group $G$ on a tree $T$ is called \emph{geometric} if there is a finite 2-complex $A$, equipped with a codimension-1 singular measured foliation, such that $\pi_1(A)=G$ and such that the action $G \curvearrowright T$ is dual to the natural action of $G$ on the space of leaves in $\tilde{A}$--the metric comes from the transverse measure (see \cite{MS84, GLP94, BF95}).  Given a geometric action $G \curvearrowright T$ with dense orbits, the Imanishi theorem implies that $G \curvearrowright T$ splits as a graph of indecomposable actions (see \cite{Gui08}).  In particular, no geometric action can contain an exceptional set.  

In \cite{CHL09}, Coulbois-Hilion-Lustig show that any $T \in \overline{cv}_n$ with dense orbits is weakly geometric in the following sense: for any basis $B$ of $F_n$, there is a canonical compact subtree $K_A \subseteq T$ such that the action $F_n \curvearrowright T_{K_A}$ dual to the restrictions of elements of $A$ to $K_A$ contains $F_n \curvearrowright T$ as its unique minimal subaction.  Here, the action $F_n \curvearrowright T$ is dual to a compact 2-complex, equipped with a codimension-1 singular measured foliation, and one might hope to generalize the theorem of Imanishi to this context, ruling out the possibility of an exceptional set in $T$.  However, Imanishi's theorem fails in this case, as is evidenced by the following example.  


This example was pointed out to us independently by M. Lustig and V. Guirardel.  Here, we use the language of relative train tracks; the reader is directed to \cite{BH92} for background.

\begin{example}\label{E.ExceptionalSet}
 Let $\alpha \in Aut(F_n)$ be represented by a relative train track map $f:G \rightarrow G$, and suppose that there are exactly two exponential strata such that the PF eiqenvalue $\lambda_l$ of the lower stratum is strictly larger than the PF eigenvalue $\lambda_u$ of the upper stratum.  Equip $G$ with a metric that restricts to the PF metrics on the exponential strata.  Let $\tilde{f}$ be a lift of $f$ to the universal cover $\tilde{G}$ of $G$; set $T_k':=\tilde{f}^k(\tilde{G})$; and finally define $T_k:=\lambda^{-k}T_k'$.  It is easy to check that the sequence $(T_k)_{k \in \mathbb{N}}$ is convergent in the Gromov-Hausdorff topology to an action $F_n \curvearrowright T \in \overline{cv}_n$.  

 Color the lower stratum green and the upper stratum red; it is evident that each action $F_n \curvearrowright T_k$ is color-preserving.  By the assumption $\lambda_l > \lambda_u$, in the limit we get an invariant (red) set that intersects a finite supporting subtree in a Cantor set.  Hence the action $F_n \curvearrowright T$ contains an exceptional set.
 \end{example}

One can check that the action $F_n \curvearrowright T$ from the example is neither indecomposable nor a graph of actions; the following result establishes a useful trichotomy.  

\begin{prop}\label{P.Tricot}
 Let $T \in \overline{cv}_n$ have dense orbits, and let $K \subseteq T$ be a finite supporting subtree.  Suppose that for each $\mu \in M_1(T)$, one has that $Supp_K(\mu)=K$.  Then one of the following holds:

 \begin{enumerate}
  \item [(i)] the action $F_n \curvearrowright T$ is indecomposable,
  \item [(ii)] the action $F_n \curvearrowright T$ splits as a graph of actions, or
  \item [(iii)] there is an exceptional subset of $T$.
 \end{enumerate}

\end{prop}

The proof of Proposition \ref{P.Tricot} uses a technique of Guirardel-Levitt \cite{GL10} that is described in the next subsection; we first present the proof, as it illustrates the effectiveness of the procedure.

\subsection{The Procedure of Guirardel-Levitt}

Let $T \in \overline{cv}_n$.  For a nondegenerate arc $I \subseteq T$, let $Y_I$ be the subtree of $T$ that is the union of all segments $J \subseteq T$ such that there are $g_1,...,g_r \in F_n$ with $J \subseteq g_1I \cup ... \cup g_rI$ and such that $g_iI \cap g_{i+1}I$ is nondegenerate.  By construction, the collection $\mathscr{Y}:=\{gY_I\}_{g \in F_n}$ is a transverse family for the action $F_n \curvearrowright T$; and the same holds for the collection $\mathscr{Y}_1:=\{g\overline{Y_I}\}_{g \in F_n}$ of translates of the closure $\overline{Y_I}$ of $Y_I$.  

\begin{proof}(Proposition \ref{P.Tricot})
Assume that the action $F_n \curvearrowright T$ is not indecomposable; there is a nondegenerate arc $I \subseteq T$ such that $Y_I \neq T$.  If $\mathscr{Y}_1$ is a transverse covering of $T$, then $T$ splits as a graph of actions by Lemma \ref{L.Graphs}.  Otherwise, the set $X_1 \subseteq K$ of points of $K$ that are not covered by trees in $\mathscr{Y}_1$ is a non-empty subset that is invariant under the pseudogroup $\Gamma_K$ generated by restrictions $g|_K:g^{-1}K \cap K \rightarrow K \cap gK$ for $g \in F_n$.  Notice that $X_1$ cannot contain a nondegenerate arc; this follows from the assumption that any ergodic $\mu \in M_1(T)$ satisfies $Supp_K(\mu)=K$.  In fact $X_1$ is nowhere dense; indeed, let $\mu \in M_1(T)$ be ergodic.  For any nondegenerate arc $J' \subseteq T$ and any $\epsilon>0$, there are $g_1,...,g_r \in F_n$ such that $\mu(K \setminus \cup_i g_iJ')<\epsilon$.  It follows that for any nondegenerate arc $J \subseteq T$, there is some tree $Y \in \mathscr{Y}$ meeting $J$ nondegenerately.  

Recursively define $\mathscr{Y}_{i+1}$ to be the collection obtained from $\mathscr{Y}_i$ via the following procedure: take unions of intersecting trees in $\mathscr{Y}_i$ to get $\mathscr{Y}_{i+1}'$, then take closures of the trees in $\mathscr{Y}_{i+1}'$ to get $\mathscr{Y}_{i+1}$.  Since $\mathscr{Y}_1$ is a transverse family, so is each $\mathscr{Y}_i$.  Put $X_i$ to be the collection of points of $K$ not covered by $\mathscr{Y}_i$; each $X_i$ is a totally disconnected, nowhere dense invariant subset.    

Note that each $Y \in \mathscr{Y}_{i+1}'$ carries the structure of a graph of actions, with vertex trees coming from $\mathscr{Y}_i$.  It may be the case that for some $i$, $\mathscr{Y}_i=\{T\}$; in this case $T$ splits as a graph of actions.  Alternatively, it could be that for some $i$, we have for $Y,Y' \in \mathscr{Y}_i$, $Y \cap Y' \neq \emptyset$ if and only if $Y=Y'$.  In this case we claim that $X_i$ is an exceptional subset of $T$.  We already know that $X_i$ is totally disconnected and nowhere dense; so we need only see that $X_i$ is perfect.  This follows from the fact that $X_i$ could contain no isolated point (the trees $Y \in \mathscr{Y}_i$ are closed).
\end{proof}

\begin{remark}\label{R.GLProcedure}
The procedure of Guirardel-Levitt has two parts:

\begin{enumerate}
 \item [(I)]  Given an action $G \curvearrowright T$ that is not indecomposable, one is able to find a subtree $Y_I \neq T$ such that the collection $\mathscr{Y}_1:=\{g\overline{Y_I}\}_{g \in F_n}$ is a transverse family for the action $G \curvearrowright T$.
 \item [(II)]  Second, given a transverse family $\mathscr{Y}_i$, one applies the iterative procedure from the above proof to obtain $\mathscr{Y}_{i+1}$:
  \begin{enumerate}
   \item [(a)] $\mathscr{Y}_{i+1}'$ consists of trees that are (maximal) unions of intersecting trees from $\mathscr{Y}_i$,
   \item [(b)] $\mathscr{Y}_{i+1}$ consists of closures of trees from $\mathscr{Y}_{i+1}'$.
  \end{enumerate}
\end{enumerate}

If the transverse family $\mathscr{Y}_1$ is ``large enough,'' then the result is either a graph of actions structure for the action $G \curvearrowright T$ or an exceptional subset of $T$.  
\end{remark}

\begin{defn}
 Let $T \in \overline{cv}_n$, and let $\mathscr{Y}_1$ be a transverse family.  If iteratively applying procedure II to $\mathscr{Y}_1$ eventually gives $\mathscr{Y}_i=\{T\}$, then we say that $T$ is obtained by an \emph{iterated graph of actions} starting from $\mathscr{Y}_1$.
\end{defn}

\subsection{Invariant Measures on Exceptional Sets}\label{SS.ExceptionalMeasure}

We now begin the analysis of trees that do not split as a graph of actions and contain an exceptional subset; we will show that such a tree $T$ admits a projection onto a tree $T'$, such that either $T'$ splits as a graph of actions, or $T'$ is indecomposable.  We begin with a definition from \cite{Lev98}, attributed there to M. Bestvina.

\begin{defn}
 Let $G$ a finitely generated group, and let $T$ an $\mathbb{R}$-tree.  An action of $G$ on $T$ by homeomorphisms is called \emph{non-nesting} if for any nondegenerate arc $I \subseteq T$ and any $g \in G$, if $gI \subseteq I$, then $gI=I$.  
\end{defn}

In \cite{Lev98} Levitt shows that if a finitely presented group $G$ admits a non-trivial, non-nesting action by homeomorphisms on an $\mathbb{R}$-tree $T$, then $G$ admits a non-trivial isometric action on some $\mathbb{R}$-tree $T_0$.  The key observation is the following:

\begin{prop}\label{P.InvariantMeasure}\cite[Proposition 4]{Lev98}
 Let $\mathscr{K}$ be a non-nesting closed system of maps on a finite tree $K$.  Assume that $\mathscr{K}$ has an infinite orbit.  Then there exists a $\mathscr{K}$-invariant probability measure $\mu$ on $\mathscr{K}$ with no atom.  
\end{prop}

A closed system of maps on a finite tree $K$ is a variant of a finitely generated pseudo-group of partial homeomorphisms of $K$, where each partial homeomorphism is required to have closed domain.  Levitt observes that if such a $\mathscr{K}$ has no infinite orbit, then the existence of a measure as in the conclusion is easy to see; this is because the system decomposes into a finite union of parallel families of finite orbits, and such a system has many invariant non-atomic measures (see \cite{Lev98}, also \cite{GLP94}).  Our goal now is to use Proposition \ref{P.InvariantMeasure} to construct invariant measures on exceptional subsets of certain trees in $\overline{cv}_n$.

We will need the following result of Guirardel.

\begin{prop}\label{P.LimitMap}\cite[Proposition 5.5]{Gui00}
 Let $T$ be a minimal non-abelian action with dense orbits of a finitely generated group $G$, and assume that $T$ is not a line.  Assume that we are given actions $T_p$, $T_p'$, and $T'$ such that $T_p \rightarrow T$ and $T_p' \rightarrow T'$; further assume that we have equivariant 1-Lipshitz maps preserving alignment $q_p:T_p \rightarrow T_p'$.  Then there is an equivariant 1-Lipshitz map preserving alignment $q:T \rightarrow T'$.
\end{prop}






As mentioned in the Acknowledgements, the proof of Proposition \ref{P.ExceptionalMeasure} was inspired by a conversation with Vincent Guirardel.  Before talking to Guirardel, we were using \cite[Theorem 3.1]{Pl75} instead of proving Proposition \ref{P.ExceptionalMeasure}.  Guirardel pointed out that \cite[Theorem 3.1]{Pl75} is not correct and mentioned that it is possible to construct an invariant measure on actions that are non-nesting and \emph{mixing}.  

The idea is as follows: given an action $F_n \curvearrowright T$ with dense orbits such that $T$ contains an exceptional set $X$, we form an action $F_n \curvearrowright T'$ on the space $T'$ obtained from $T$ by collapsing the components of the complement of $X$ to points.  One checks that $F_n \curvearrowright T'$ is a non-nesting action on a tree $T'$.  We then find an invariant, non-atomic measure on $T'$, and then this measure pulls back to an invariant measure on $T$, supported on $X$.  

The idea of Guirardel ensures the existence of an invariant, non-atomic measure as long as the action $F_n \curvearrowright T'$ is mixing.  However, we cannot assume that our quotient action is mixing; a resultant problem is that our attempt to build an invariant measure may result in an atomic measure, which may not be invariant.  Fortunately, in this case we are able to use Proposition \ref{P.LimitMap} to obtain a graph of actions structure on $F_n \curvearrowright T$.

It should be noted that this ``collapse, measure, pull-back'' idea is not necessary to prove Proposition \ref{P.ExceptionalMeasure}; indeed, it is possible to work in a pseudogroup associated to the original tree, and this approach is likely more intuitive.  However, we have chosen not to use pseudogroups so as to not further lengthen an already long exposition.  

\begin{prop}\label{P.ExceptionalMeasure}
 Let $T \in \overline{cv}_n$ have dense orbits.  Suppose that $T$ does not split as a graph of actions and that $T$ contains an exceptional set $X$.  Then there is $\mu \in M_0(T)$ supported on $X$.  
\end{prop}

\begin{proof}
 Let $T$ and $X$ as in the statement, and fix a basis $B=\{g_1,...,g_n\}$ for $F_n$.  Let $x \in T$, and set $K$ to be the convex hull of $\{gx|g \in B^{\pm}\}$, so $K$ is a finite supporting subtree.  

 Define a relation $R_0$ on $T$ by $xR_0y$ if $[x,y]$ is contained in the closure of some component of $T \setminus X$.  Let $R$ be the equivalence relation generated by $R_0$.  We want to see that the classes of $R$ are closed subtrees of $T$.  Let $Y$ be some class of $R$, and let $y_i \in Y$ be a sequence of points converging to $y \in T$.  The segments $I_m:=\cap_{k \geq m} [y_1,y_k]=[y_1,p_m]$ are contained in $Y$, and $p_m$ converges to $y$.  If $y \notin Y$, then $y \in X$.  For $k >> 0$, there is $g \in F_n$ such that $g[p_m,y] \subseteq K$; as $X \cap K$ is a Cantor set, by increasing $m$ if necessary, we can assume that $g[p_m, y] \cap X=\{y\}$, hence $y \in Y$.  

 The set $\mathscr{T}=\{Y_i\}_{i \in I}$ of non-degenerate classes of $R$ is clearly a transverse family; further, it is easy to see that for $Y \neq Y' \in \mathscr{T}$, $Y \cap Y' = \emptyset$.  Put $T'$ to be the quotient of $T$ by $R$; the natural map $f:T \rightarrow T'$ is continuous on segments of $T$ and has convex point preimages.  As classes of $R$ are closed, $T'$ is a regular Hausdorff space.  Further, it is easy to see that $T'$ is uniquely arc connected and locally arc connected, hence by \cite{MO90} $T'$ is an $\mathbb{R}$-tree.  There is by construction an action of $F_n$ on $T'$ such that each $f \in F_n$ acts as a homeomorphism on any finite subtree of $T'$ .  It is easy to check that the action is non-nesting and supported on the finite subtree $K':=f(K)$.



 Let $K_i \subseteq T$ be the convex hull of $\{gx: ||g||_B \leq i\}$, so $K_1=K$, and the sequence $K_i$ is an invasion of $T$ by finite subtreees.  Let $T_i$ be the geometric action corresponding to the restrictions of elements of $B$ to $K_i$; see \cite{GLP94}.  The set $X \cap K$ gives rise to an exceptional sets $X_i \subseteq T_i$; as above, let $f_i:T_i \rightarrow T_i'$ be the quotient map, so that we have non-nesting actions $F_n \curvearrowright T_i'$ that are dual to finite systems of maps as in \cite{Lev98}.  By Proposition \ref{P.InvariantMeasure}, each such action supports an invariant measure with no atoms, so we get actions by isometries $F_n \curvearrowright T_i''$ along with equivariant, alignment preserving maps $f_i'':T_i \rightarrow T_i''$.  Passing to a subsequence if necessary and rescaling, we get a sequence of actions $F_n \curvearrowright T_i''$ that is convergent in $\overline{cv_n}$ to some action $F_n \curvearrowright T''$.  

 Pulling back via $f_i''$ the Lebesgue measure on $T_i''$, we get invariant measures $\mu_i$ on $T_i$ supported on $X_i$; let $Y_i$ be the tree with underlying set $T_i$ and Lebesgue measure $\mu_L(Y_i):=\mu_L(T_i) + \mu_i$.  We get equivariant 1-Lipschitz, alignment preserving maps $g_i:Y_i \rightarrow T_i$ and $h_i:Y_i \rightarrow T_i''$; further, as each $\mu_i$ is non-atomic, each $g_i$ is a bijection.  

 For an $F_n$-tree $U$ and an element $g \in F_n$, let $A_U(g)$ denote the characteristic set of $g$ in $U$.  Let $g \in F_n$ be hyperbolic in $T$; then for $i >>0$, $g$ is hyperbolic in $T_i$.  It follows from \cite[Lemma 5.1]{Gui00} that $f_i''(A_{T_i}(g))=A_{T_i''}(g)$; hence it follows that $l_{Y_i}(g)=l_{T_i}(g) + l_{T_i''}(g)$.  On the other hand, we have arranged that the sequences $l_{T_i}(g)$ and $l_{T_i''(g)}$ are convergent; hence, the sequence $l_{Y_i}(g)$ is convergent.  It follows that the sequence of actions $F_n \curvearrowright Y_i$ converges to an action $F_n \curvearrowright Y \in \overline{cv}_n$; further, Proposition \ref{P.LimitMap} gives 1-Lipschitz alignment preserving maps $g:Y \rightarrow T$ and $h:Y \rightarrow T''$.  

 If the map $g:Y \rightarrow T$ is a bijection, then we are finished; indeed, in this case $\mu_L(Y)-g^*(\mu_L(T))$ is an invariant measure with no atoms supported on $X$.  Hence, we suppose that the sequence $\mu_i$ converges to an atomic measure $\nu$ on $K'$, and let $y \in K'$ have $\nu(\{y\})=m>0$.  It is easy to see that in this case there is a germ $\hat{d}$ of a direction $d$ at $y$ in $K'$ such that the set $O(\hat{d}):=\{g \in F_n| gy \in K'$ and $g[y,\hat{d}] \cap K'$ is non-degenerate$\}$ is finite (else, the sequence $\mu_i(K')$ is unbounded).  It follows that there is a positive lower bound for the translation lengths of hyperbolic elements $g \in F_n$ such that $A(g)$ meets $F_ny$ and non-degenerately interects $F_nd$.  This shows that $T''$ does not have dense orbits, \emph{i.e.} $T''$ has non-empty simplicial part.  

 As the map $h:Y \rightarrow T''$ is equivariant, 1-Lipschitz, and alignment-preserving, $Y$ has a non-empty simplicial part as well.  Let $\mathscr{C}=\{Y_i\}_{i \in I}$ be the transverse covering of $Y$ by subtrees $Y_i$, which are either simplicial edges or have dense orbits (see Proposition \ref{C.IndiscreteGraph}); and put $\mathscr{C}_0 \subseteq \mathscr{C}$ to be collection of trees with dense orbits.  Evidently, $g(\mathscr{C}_0):=\{g(Y_i)|Y_i \in \mathscr{C}_0\}$ is a transverse covering of $T$, hence by Lemma \ref{L.Graphs} $T$ splits as a graph of actions, a contradiction.



\end{proof}

This immediately gives.

\begin{cor}\label{C.NoException}
 Let $T \in \overline{cv}_n$ have dense orbits, and let $K \subseteq T$ be a finite supporting subtree.  Suppose that $T$ does not split as a graph of actions.  If for each $\mu \in M_0(T)$, one has that $Supp_K(\mu)=K$, then $T$ is indecomposable.
\end{cor}

\begin{proof}
 By the hypothesis on $M_0(T)$ and by Proposition \ref{P.ExceptionalMeasure}, we have that $T$ contains no exceptional set; the result is then a restatement of Proposition \ref{P.Tricot}, 
\end{proof}

To finish this section, we provide an example showing that the hypothesis of Proposition \ref{P.ExceptionalMeasure} is necessary and not an artifact of our proof.

\begin{example}\label{E.HiddenDynamics}
Let $F_n \curvearrowright T$ be the action constructed in Example \ref{E.ExceptionalSet}; as aforementioned, $F_n \curvearrowright T$ is neither indecomposable nor a graph of actions.  By Proposition \ref{P.ExceptionalMeasure} we can find an invariant measure $\nu$ supported on the exceptional set $E \subseteq T$; we remark that it is easy to check that $E$ is $\mu_L$-measure zero, where $\mu_L=\mu_L(T)$ is the Lebesgue measure for $T$.  Further, one can show that $dim(M_0(T))=2$, \emph{i.e.} $\mu_L$ and $\nu$ are the only ergodic measures on $T$ up to rescaling; see \cite{CH10}. 

Let $T'$ be the tree with underlying set $T$ and Lebesgue measure $\mu_L(T')=\mu_L +\nu$; as $\nu$ is non-atomic, the ``identity map'' $f:T \rightarrow T'$ is continuous on segments and bijective.  By Lemma \ref{L.Observers}, $f$ extends to a unique homeomorphism $\hat{f}:\hat{T} \rightarrow \hat{T'}$; refer to Subsection \ref{SS.Observe}.  On the other hand, since $\nu$ is singular with respect to $\mu_L(T)$, the map $f$ is not continuous with respect to the metric topologies on $T$ and $T'$; we leave this as an exercise to the reader.

As in Subsection \ref{SS.Observe}, let $q \in T$ be a base point, and let $(p_k)$ be a sequence in $T$.  Put $I_m:= \cap_{k \geq m} [q,p_k]$, so $I_m=[q,r_m]$, and we have $I_m \subseteq I_{m+1}$.  Recall that the \emph{inferior limit} of $(p_k)$ from $q$ is the limit $\lim_q p_k:= \lim r_m$.  If $(p_k)$ is convergent in the metric topology to $p \in T$, then $p=\lim_q p_k$.  Since $f$ is not continuous with respect to the metric topology, we can find a convergent sequence $(p_k)$ of points in $T$ such that $(f(p_k))$ is not convergent in $T'$.  Since $f$ is continuous on segments and since $\lim p_k = \lim_q p_k =p$, we have that the sets $I_m':=[f(q),f(r_m)]$ satisfy $\overline{\cup_m I_m'}=f([q,p])$; it follows that the distances $d_{T'}(f(p_k),[f(q),f(p)])$ are bounded below by some number $c>0$.  On the other hand, since $\lim p_k=\lim_q p_k=p$, we have that $d_T(p_k,[q,p])\rightarrow 0$.

Replace $(p_k)$ be a subsequence to ensure that $\sum_k d(p_k, [q,p])$ is finite.  Let $y_k \in [q,p]$ be defined so that $d(p_k,y_k)=d(p_k,[q,p])$, and set $J_k:=[p_k,y_k]$.  By the above paragraph, the lengths of $f(J_k)$ are bounded below by $c >0$.  It is an easy exercise using the fact that $T$ has dense orbits to use the intervals $J_k$ to produce a finite length ray $\rho$ in $T$, such that $f(\rho)$ is unbounded in $T'$. Since $f$ is continuous on segments, and since $\hat{f}$ is a homeomorphism, we can conclude that $\rho$ converges to a point $w \in \overline{T} \setminus T$, while $f(\rho)$ converges to a point of $\partial T'$.  

Let $w'\in T$ be a point with trivial stabilizer; we are going to form an ``HNN-extension'' of $T$ using $w$ and $w'$.  Let $S$ be the Bass-Serre tree for the splitting $F_{n+1}=F_n \ast_{\langle 1 \rangle}$.  Associate to each vertex of $S$ is a copy of the tree $T$.  Let $\tau \subseteq S$ be a lift of a spanning tree of $S/F_{n+1}$, \emph{i.e.} $\tau$ is an edge $e$ of $S$; put $p_e:=w$ and $p_{\overline{e}}:=w'$, and extend this equivariantly to associate to each directed edge of $S$ an attaching point to get a graph of actions $\mathscr{G}$.  Let $Y:=T_{\mathscr{G}}$ be the tree dual to $\mathscr{G}$. 

By construction, the measure $\nu$ on $T$ does not give an invariant measure on $Y$; this is because we have arranged that any finite arc of the form $[z,w]$ would have infinite measure.  On the other hand, any invariant measure $\mu'$ on $Y$ clearly gives rise to an invariant measure on $T$.  Hence, $dim(M_0(Y))=1$, \emph{i.e.} $Y$ is uniquely ergodic.  The existence of the exceptional set $E \subseteq T$ gives rise to an exceptional set $E' \subseteq Y$, and it follows that there is no measure supported on $E'$.   
\end{example}

It should be noted that the trick in the example can be applied to any tree $T$ with dense orbits such that there exist two mutually-singular ergodic measures $\mu, \mu' \in M_0(T)$ such that $Supp(\mu) \subseteq Supp(\mu')$; the result will be an HNN-extension of $T$ in which the measure $\mu$ has been ``hidden.''  

\section{Dynamics on $\overline{CV}_n$}\label{S.MainDynamics}

We are now in a position to prove our main dynamics result; we are left to handle convergence for trees $T \in \overline{cv}_n$ that are not indecomposable and do not split as a graph of actions.  The results of Section \ref{S.Measures} shows that in this case, there is a projection $T'$ of $T$ that is either indecomposable or a graph of actions.  We are left to show that this projection does not distort the tree $T$ in non-obvious ways: we need to control $L^2(T')$.

\begin{lemma}\label{L.ControlL2}
 Let $T \in \overline{cv}_n$ have dense orbits, and let $\mu \in M_0(T)$.  Then $L^2(T) \subseteq L^2(T_{\mu})$.
\end{lemma}

\begin{proof}
 Let $T$ as in the statement, and note that if $\mu$ is absolutely continuous with respect to $\mu_L$ then the result is obvious.  Hence, we may reduce to the case that $\mu$ is ``generic'' in the sense that there is no $\nu' \in M_0(T)$ singular with respect to $\mu$; any generic $\mu$ can be rescaled so that $\mu_L \leq \mu$, so we assume that $\mu_L \leq \mu$.  Hence, a 1-Lipschitz, alignment-preserving map $f:T_{\mu} \rightarrow T$, and we need to establish that $L^2(T_{\mu})=L^2(T)$.  It is easy to check that $f$ is a bijection that is continuous on segments; therefore, by Lemma \ref{L.Observers}, $f$ induces a homeomorphism $\hat{f}:\hat{T_{\mu}} \rightarrow \hat{T}$ so that $L^2(T)=L^2(T_{\mu})$.  
\end{proof}

This immediately gives:

\begin{cor}\label{C.OthersConverge}
 Let $T \in \overline{cv}_n$ have dense orbits, and suppose that $T$ is not indecomposable and does not split as a graph of actions.  There is a projection $f:T \rightarrow T'$ such that:

 \begin{enumerate}
  \item [(i)] either $T'$ is indecomposable, or $T'$ splits as a graph of actions,
  \item [(ii)] $L^2(T) \subseteq L^2(T')$.
 \end{enumerate}

\end{cor}

\begin{proof}
Let $T$ as in the statement; choose $\mu \in M_0(T)$ such that $Supp(\mu)$ is set-theoretically minimal.  It follows from Propositions \ref{P.SuppProperties} and \ref{P.ExceptionalMeasure} that if $T':=T_{\mu}$ contains an exceptional set, then $T'$ splits as a graph of actions; hence, by Corollary \ref{C.NoException}, either $T'$ is indecomposable or $T'$ splits as a graph of actions.  Further, it follows from Lemma \ref{L.ControlL2} that $L^2(T) \subseteq L^2(T')$.  
\end{proof}

We are now prepared to prove the main dynamics result of the paper.

\begin{theorem}\label{T.Main}
 Let $\phi:F_n \rightarrow F_n$ be an irreducible endomorphism, and suppose that $\phi \notin Aut(F_n)$.  Further suppose that for any $T \in \overline{cv}_n$, $T\phi$ is non-trivial.  Then $\phi$ acts on $\overline{CV}_n$; there is a unique fixed point $[T_{\Phi}] \in CV_n \subseteq \overline{CV}_n$; and for any compact neighborhood $N$ of $[T_{\Phi}]$, there is $k=k(N)$ such that for any $[T] \in \overline{CV}_n$, $[T]\phi^k \in N$.  
\end{theorem}

\begin{proof}
 Let $\phi$ as in the statement.  First note that it follows from Propositions \ref{P.SimpConverge} and \ref{P.IndecompConverge}, Lemma \ref{L.GraphsConverge}, and Corollary \ref{C.OthersConverge} that for any $T \in \overline{cv}_n$, we have that $[T]\phi^k \rightarrow [T_{\Phi}]$, hence $[T_{\Phi}]$ is the unique fixed point of $\phi$ acting on $\overline{CV}_n$.  

 Toward contradiction suppose that there is a compact neighborhood $N$ of $[T_{\Phi}]$ in $\overline{CV}_n$ and actions $[T_k]$ such that $[T_k]\phi^k \notin N$.  Note that $\phi$ induces a continuous function on $\overline{CV}_n$.  It follows that the set of accumulation points $\{[T_k]\}_{k \in \mathbb{N}}$ is $\phi$-invariant and does not contain $[T_{\Phi}]$, hence this set must be empty, a contradiction.
\end{proof}

To finish this section we bring an immediate corollary of Theorem \ref{T.Main}, showing the existence of a certain type of rigid subgroup of $F_n$.  

\begin{cor}\label{C.Rigid}
For any $C>1$, there is a finitely generated, non-abelian subgroup $H \leq F_n$, such that for any non-trivial $h,h' \in H$ and any trees $T,T' \in \overline{cv}_n$, one has $l_T(h)>0$ and $$\frac{1}{C} \leq \frac{l_T(h)/l_T(h')}{l_{T'}(h)/l_{T'}(h')} \leq C$$
\end{cor}

\begin{proof}
 Let $\phi:F_n \rightarrow F_n$ be an irreducible, non-surjective, admissible endomorphism, and set $H_k:= \phi^k(F_n)$.  It is immediate from Theorem \ref{T.Main} that for any $C>1$, there is $K$ such that for all $k \geq K$, $H_k$ satisfies the desired properties.  
\end{proof}

\section{Structure of Actions in $\overline{cv}_n$}\label{S.Decompose}

In this section we expand the results of Section \ref{S.Measures} to give a dynamical decomposition of trees in $\overline{cv}_n$ that generalizes the dynamical decomposition of geometric trees coming from Imanishi's theorem.  The results we obtain are similar in spirit to new work of Guirardel-Levitt \cite{GL10}.  

\subsection{The Transverse Family $\mathscr{F}$}\label{SS.F}

Suppose that $T \in \overline{cv}_n$ is not simplicial; if $T$ does not have dense orbits, then by Proposition \ref{C.IndiscreteGraph} $T$ splits as a graph of actions, with vertex trees either simplicial edges or trees with dense orbits.  So, let $T \in \overline{cv}_n$ have dense orbtis.  By Proposition \ref{P.M0T} we have that $M_0(T)$ is $\mathbb{R}_{>0}$-spanned by a finite set $B=\{\nu_1,...,\nu_r\}$ of mutually-singular ergodic measures, and by Proposition \ref{P.SuppProperties} we have for $\nu_i,\nu_j \in B$ with $Supp(\nu_i),Supp(\nu_j)$ non-degenerate, if $Supp(\nu_i) \neq Supp(\nu_j)$, then for any finite tree $K \subseteq T$, $Int(Supp_K(\nu_i)) \cap Int(Supp_K(\nu_j)) = \emptyset$.  We define a family $\mathscr{F}$ of subtrees of $T$; a subtree $Y \subseteq T$ is a member of $\mathscr{F}$ if:

\begin{enumerate}
 \item [(i)] $Y$ is non-degenerate,
 \item [(ii)] there is $\nu_i \in B$ and an invasion $\{F_k\}_{k \in \mathbb{N}}$ of $Y$ by finite subtrees $F_k$ such that $F_k = Supp_{F_k}(\nu_i)$, and 
 \item [(iii)] $Y$ is maximal with respect to (ii).
\end{enumerate}

It is then easy to check that $\mathscr{F}$ is a transverse family for the action $F_n \curvearrowright T$; further, by Proposition \ref{P.SuppProperties}, there are finitely-many $F_n$-orbits of trees in $\mathscr{F}$.  

\begin{lemma}\label{L.SuppGraph}
 Let $T \in \overline{cv}_n$ have dense orbits, and suppose that there are $\mu,\mu' \in M_0(T)$ such that $Supp(\mu)\neq Supp(\mu')$ are both non-degenerate.  Then $T$ splits as a graph of actions.
\end{lemma}

\begin{proof}
Let $\mathscr{F}$ be the transverse family defined above.  Since there are $\mu, \mu' \in M_0(T)$ with $Supp(\mu) \neq Supp(\mu')$ both non-degenerate, the family $\mathscr{F}$ contains more than one tree.  We apply the iterative procedure II of Remark \ref{R.GLProcedure} to get a sequence $\mathscr{F}_k$ of transverse families in $T$.  This procedure might terminate with an exceptional set $E \subseteq T$; in this case, by Proposition \ref{P.ExceptionalMeasure}, either $T$ splits as a graph of actions, or there would be ergodic $\nu_i \in B$ supported on $E$.  By Proposition \ref{P.SuppProperties}, $Supp(\nu_i)$ would be contained in the interior of $Supp(\nu_j)$ for some $\nu_j \in B$ with non-degenerate support; but this is impossible by definition of $\mathscr{F}$ and by construction of $E$.  So, for some $k_0$, we have that $\mathscr{F}_{k_0}=\{T\}$; suppose that $k_0$ is minimal with respect to this property.  Then $T$ splits as a graph of actions corresponding to the transverse covering $\mathscr{F}_{k_0-1}$.
\end{proof}

The proof of Lemma \ref{L.SuppGraph} also shows that if for any finite arc $I \subseteq T$, $I$ is the union of finitely many subintervals $I_1,...,I_k$ such that $I_j = Supp_{I_j}(\nu_{i_j})$, then the set orbits of vertex trees in this graph of actions structure for $T$ bijectively corresponds to the set of non-degenerate support sets of the ergodic measures $\nu \in M_0(T)$.  In particular, this would be the case if the action $F_n \curvearrowright T$ happened to be geometric.

\begin{lemma}\label{L.Supp1}
 Let $T \in \overline{cv}_n$ have dense orbits, and let $\mathscr{F}$ be the transverse family constucted above.  Suppose that $\mathscr{F}$ contains at least two orbits of trees.  Then for each $Y \in \mathscr{F}$, the set-wise stabilizer $Stab(Y)$ is a vertex group of a very small splitting of $F_n$.  
\end{lemma}

\begin{proof}
 Let $T$ and $\mathscr{F}$ as in the statement.  By hypothesis, we can find ergodic $\nu, \nu' \in M_0(T)$ such that $Supp(\nu)$ and $Supp(\nu')$ are non-degenerate and such that $Supp(\nu) \neq Supp(\nu')$.  Let $\mu_L=\nu_1 + ... + \nu_l$ be the decomposition of the Lebesgue measure on $T$ as a sum of muntually-singular ergodic measure $\nu_i \in M_0(T)$.  Define a measures $\mu \in M_1(T)$ by $\mu:= \Sigma_{Supp(\nu_i)=Supp(\nu)} \nu_i$; we get an equivariant 1-Lipschitz alignment-preserving map $f:T \rightarrow T_{\mu}$.  We observe that if $Y \in \mathscr{F}$ corresponds to $Supp(\nu')$, then the image of $Y$ under $f$ is a point.  It follows that $Stab(Y)$ fixes a point $y \in T_{\mu}$.  By \cite{Gui98} $Stab(Y)$ is contained in a vertex group of a very small splitting of $F_n$.  

 On the other hand, it is immediate that $Stab(\{y\})=Stab(\overline{Y})$.  Further, it is clear that if $f \in F_n$ is hyperbolic in $T$, then $f \in Stab(Y)$ if and only if $f \in Stab(\overline{Y})$.  Toward a contradiction suppose that there is $g \in Stab(\overline{Y}) \setminus Stab(Y)$, then $g$ acts elliptically on $\overline{Y}$, and so $g$ must fix a point $p \in \overline{Y} \setminus Y$.  There is only one direcction (in $\overline{Y}$) at $p$, so $g$ must fix a non-degenerate arc $[p',p] \subseteq \overline{Y}$.  We show this is impossible.

 Let $\eta \in M_0(T)$ ergodic such that $Supp(\eta)=Y$, and let $I_0 \subseteq Y$ be a small arc.  By ergodicity, for any $J \subseteq Y$ and $\epsilon > 0$, there are $g_1,...,g_r \in F_n$ such that $\eta(J \setminus g_1I_0 \cup ... \cup g_rI_0) < \epsilon$.  It easily follows that there are elements $g' \in Stab(Y)$ with arbitrarily short translation length, hence $Stab(Y)$ acts on $Y$ with dense orbits, and the same holds for the action $Stab(\overline{Y}) \curvearrowright \overline{Y}$.  This contradicts the above observation that $g$ must fix a non-degenerate arc of $\overline{Y}$ and completes the proof.  
\end{proof}

\subsection{The Case $\mathscr{F}=\{T\}$}

We turn to analyzing trees in $\overline{cv}_n$ for which the family $\mathscr{F}$ is as simple as possible.  

\begin{prop}\label{P.Diagram}
 Let $T \in \overline{cv}_n$ have dense orbits, and suppose that for each $\nu \in M_0(T)$, if $Supp(\nu)$ is non-degenerate, then $Supp(\nu)=T$.  There is a finite set of projections $\{P_i\}_{i \in \{1,...,r\}}=\{f_i:T \rightarrow T_i\}_{i \in \{1,...,r\}}$ such that:
 \begin{enumerate}
  \item [(i)] $dim(M_0(T_i)) < dim(M_0(T))$,
  \item [(ii)] there is a partial order $\leq$ on $\{P_1,...,P_r\}$ such that:
    \begin{enumerate}
     \item [(a)] if $P_i \leq P_j$, then $f_i$ factors through $f_j$,
     \item [(b)] if $P_i < P_j$, then $dim(M_0(T_i)) < dim(M_0(T_j))$,
     \item [(c)] for $P_i$ minimal, if $T_i$ contains an exceptional set, then $T$ (and $T_i$) splits as a graph of actions.
    \end{enumerate}
  \item [(iii)] for any projection $T \rightarrow T'$, there is $T_i$ and is $\mu \in M_0(T_i)$ such that $T'$ is equivariantly isometric to $(T_i)_{\mu}$ and such that the natural map $T_i \rightarrow (T_i)_{\mu}$ is a bijection.
 \end{enumerate}
\end{prop}

\begin{proof}
 Let $T$ as in the statement.  We first arrange that $\mu_L$ is contained in the interior of $M_0(T)$, \emph{i.e.} there is no $\nu \in M_0(T)$ singular with respect to $\mu_L$.  This is accomplished by replacing $T$ with $T_\mu$ for some $\mu \in M_0(T)$ with the desired property; the natural map $T \rightarrow T_{\mu}$ is a bijection.  By Proposition \ref{P.M0T} $M_0(T)$ is spanned by a finite set $\{\nu_1,...,\nu_r\}$ of mutually-singular ergodic measures.  In particular, there are finitely many sets $Supp(\nu)$ for ergodic $\nu \in M_0(T)$; let $\mathscr{X}=\{X_1,...,X_s\}$ denote the collection of these support sets.  Then $\mathscr{X}$ carries the obvious partial order $\leq$, with unique maximal element $X_i =T$.  By Proposition \ref{P.ExceptionalMeasure}, if $X_j=Supp(\nu)$ is minimal with respect to $\leq$, then if $T_{\nu}$ contains an exceptional, then $T_{\mu}$ splits as a graph of actions.  Define a set of measures $\mathscr{M}=\{\mu_1,...,\mu_s\} \subseteq M_0(T)$ by $\mu_i:=\mu_L|_{X_i}$.  It is then easy to check that the collection $\{f_i:T \rightarrow T_{\mu_i}|Supp(\mu) \neq T\}$ satisfy the conclusions of the proposition.
\end{proof}

Let $T \in \overline{cv}_n$ have dense orbits, and assume that no exceptional subset of $T$ has non-zero measure.  For any ergodic $\mu \in M_1(T)$, and for any finite arc $I \subseteq T$ with $Supp_I(\mu)$ non-degenerate, the restriction $\mu|_I$ determines $\mu$ in the sense that $\mu$ is the unique minimal element of $M_1(T)$ restricting to $\mu|_I$ on $I$.  To extend the analogy between length measures and measured laminations on surfaces, we look for invariant subsets of $T$ that encode $\mu$ as above.  As any two non-degenerate, invariant subsets supporting $\mu$ must intersect non-degenerately, we are naturally led to consider transverse families with one orbit of trees.  We need the following result, whose proof was sketched to us by V. Guirardel.


\begin{lemma}\label{L.FF}\cite{GL10}
 Let $T \in \overline{cv}_n$ have dense orbits.  Let $Y \subseteq T$ be a non-degenerate subtree such that for all $g \in F_n$, either $gY=Y$ or $gY \cap Y = \emptyset$.  Then $Stab(Y)$ is a free factor of $F_n$.
\end{lemma}

\begin{proof}
 We prove the result in the case that $\mathscr{F}=\{T\}$, as the general case follows immediately from this case.  First note that $Stab(Y)$ is non-trivial.  Indeed, let $I \subseteq Y$ be a non-degenerate arc, then there is ergodic $\nu \in M_1(T)$ such that $\nu(I)>0$.  For a small subarc $I_0 \subseteq I$ with $\nu(I_0)>0$ and any $\epsilon > 0$, we can find elements $g_1,...,g_r \in F_n$ such that $\nu(I \setminus \cup_i g_iI_0) > \nu(I) - \epsilon$.  It follows that there is some $g \in F_n$ such that $A(g) \cap I$ contains a fundamental domain for the action $\langle g \rangle \curvearrowright A(g)$, and since $Y$ is disjoint from its translates, it follows that $g \in Stab(Y)$.

 Fix a basis $B=\{x_1,...,x_n\}$ for $F_n$ and a point $y \in T$; define $K_i$ to be the convex hull of the set $\{gy:||g||_B \leq i\}$; and let $T_i$ be the geometric action dual to the restrictions of elements of $B$ to $K_i$.  So, $T$ is the strong limit of the sequence $T_i$, and as $T$ has trivial arc stabilizers, each $T_i$ has trivial arc stabilizers as well.  

 Replace $Y$ with a translate if necessary to ensure that $Y \cap K_1$ is non-degenerate.  Define $Y_i \subseteq T_i$ as follows: $Y_i^1:= Y \cap K_i$; $Y_i^{r+1}$ is the union of $Y_i^r$ and all translates of $Y_i^r$ meeting $Y_i^r$; $Y_i':= \cup_r Y_i^r$; and $Y_i:= \overline{Y_i'}$.  Hence $Y_i$ is a closed subtree of $T_i$ disjoint from its translates.  Being geometric, $T_i$ splits as a graph of actions, with each vertex tree either simplicial edge or an indecomposable tree.  Evidently, if $Y_i$ meets an indecomposable vertex tree $V$ in a non-degenerate arc, then $V \subseteq Y_i$.  By the hypothesis $\mathscr{F}=\{T\}$, we have that for any non-degenerate $I,J \subseteq T$, there is $g \in F_n$ such that $gI \cap J$ is non-degenerate.  It follows that the splitting of each $T_i$ into a graph of indecomposable trees and simplicial edges can contain at most one orbit of indecomposable trees; further, since $T$ is a strong limit of the $T_i$'s, if $V$ is an indecomposable vertex tree of some $T_i$, then there is $j \geq i$ such that $V \subseteq Y_j$.  

 We collapse to a point each tree in the orbit of $Y_i$ tree to get a tree $S_i$, equipped with a non-trivial action of $F_n$.  From the above discussion, $S_i$ is a simplicial tree with trivial arc stabilizers such that $Stab(Y_i)$ is a vertex stabilizer in $S_i$.  Hence, $Stab(Y_i)$ is a free factor of $F_n$.  To conclude, just note that if $g_1,...,g_k \in Stab(Y)$ are hyperbolic, then there is $i_k$ such that each $g_j$ is hyperbolic in $T_{i_k}$ and such that $A(g_j) \subseteq Y_{i_k}$; therefore eventually $Stab(Y_i)=Stab(Y)$.
\end{proof}

\begin{defn}\label{D.Mixing}
An action $T \in \overline{cv}_n$ is called \emph{mixing} if for any non-degenerate arcs $I,J \subseteq T$, there are $g_1,...,g_r \in F_n$ such that $J \subseteq g_1I \cup ... \cup g_rI$.  
\end{defn}

Mixing differs from indecomposable in that we place no requirements on the overlaps $g_iI \cap g_{i+1}I$.  Mixing is equivalent to the following condition: for any non-degenerate arcs $I,J \subseteq T$, $J$ can be subdivided into arcs $J_1,...,J_r$ such that there are $g_1,...,g_r \in F_n$ with $g_iJ_i \subseteq I$.  

\begin{lemma}\label{L.MixLemma}
Let $T \in \overline{cv}_n$.  The following are equivalent

\begin{enumerate}
 \item [(i)] For any direction $d$ at $x \in T$, and any non-degenerate arc $I \subseteq T$ there is $g \in F_n$ such that $gx \in I$ and $gd \cap I$ is non-degenerate,
 \item [(ii)] the action $F_n \curvearrowright T$ is mixing. 
\end{enumerate}

\end{lemma}

\begin{proof}
 To see (i) implies (ii), let $T \in \overline{cv}_n$, and assume that $T$ is not mixing.  There are non-degenerate arcs $I,J \subseteq T$ such that $J \setminus \cup_{g \in F_n}gI \neq \emptyset$; let $y \in J \setminus \cup_{g \in F_n}gI$, and let $d$ be a direction at $y$ meeting $J$ non-degenerately.  Evidently, there is no $g \in F_n$ such that $gy \in I$ and such that $gd \cap I$ is non-degenerate.

 To see (ii) implies (i), assume that $T$ is mixing, and let $I \subseteq T$ be any non-degenerate arc.  For any direction $d$ at $x \in T$, we take $y \in d \subseteq T$, so $[x,y]$ is a non-degenerate arc.  Since $T$ is mixing $[x,y]$ can be divided into finitely many subarcs $[x,y]=[x=y_0,y_1] \cup [y_1,y_2] \cup ... \cup [y_{r-1},y_r]$ such that there are $g_1,...,g_r \in F_n$ with $g_i[y_{i-1},y_i] \subseteq I$.  Hence, $g_ox \in I$, and $g_0d \cap I$ is non-degenerate.  
\end{proof}

\begin{prop}\label{P.MixingSubaction}
 Suppose that $T \in \overline{cv}_n$ has dense orbits.  There is a transverse family $\mathscr{T}=\{Y_i\}_{i \in I}$, with orbits in $\mathscr{T}$ in bijective correspondence with orbits in $\mathscr{F}$, such that the actions $Stab(Y_i) \curvearrowright Y_i$ are mixing.  
\end{prop}

\begin{proof}
 Let $T$ as in the statement; we suppose that for any $\mu \in M_0(T)$ with non-degenerate support, we have $Supp(\mu)=T$, as the general case follows immediately from this case by considering the trees in $\mathscr{F}$.  We proceed by induction on $n$.  It follows from Harrison's theorem that every $T \in \overline{cv}_2$ is geometric.  The result is clear in this case; indeed, it follows from Imanishi's theorem that any geometric tree with dense orbits is a graph of indecomposable actions.  So, we suppose the result holds for all $T \in \overline{cv}_m$ for $m < n$.

 If $T \in \overline{cv}_n$ contains an exceptional set, then we are done by Lemma \ref{L.FF} and induction, so we suppose this is not the case.  It follows that for each $\mu \in M_0(T)$, $Supp(\mu)=T$, hence applying Procedure II of Remark \ref{R.GLProcedure} to any transverse family in $T$ will eventually produce $T$.  


 Assume that there is some $y \in T$ and direction $d$ at $y$ such that for any finite arc $J \subseteq T$, the set $\{g \in F_n| gy \in J$ and $gd \cap J$ is non-degenerate$\}$ is finite.  In this case, we may blow-up the orbit of $d$ as follows; split $T$ open at $y$, gluing directions not in the orbit of $d$ back to $y$, and for each $d'$ at $y$ in the orbit of $d$, glue a simplicial edge of length 1 to $y$; finally glue each direction in the orbit of $d$ to its corresponding simplicial edge.  This gives a tree $T'$, which, by the finiteness assumption on $d$, carries an isometric action of $F_n$.  

 The obvious map $f:T' \rightarrow T$ is equivaraint, 1-Lipschitz and alignment-preserving.  The graph of actions structure on $T'$ gives a graph of actions structure $T$.  Further, collasping to a point every tree with dense orbits in $T'$ gives an action of $F_n$ on a simplicial tree with trivial arc stabilizers, where vertex stabilizers correspond to stabilizers of the trees in $T'$ with dense orbits.  Hence, we get a transverse covering of $T$ by subtrees $\{Y_i\}_{i \in I}$ such that $Stab(Y_i)$ is a free factor of $F_n$, and the result follows from induction.

 We are left to consider the case that no direction $d$ in $T$ satisfies the above finiteness condition.

 \begin{claim}
  Let $T \in \overline{cv}_n$, and suppose that there is no direction $d$ at $x \in T$ such that for every non-degenerate arc $I \subseteq T$, $\{y|y=gx$ and $gd \cap I$ is non-degenerate$\}$ is finite.  Then $T$ is mixing.
 \end{claim}

  \begin{proof}
   If it were the case that for any direction $d$ at $y$ in $T$ and any non-degenerate arc $J \subseteq T$, the collection $\{gy|gd \cap J$ is non-degenerate$\}$ is dense in $J$, the $T$ would be mixing by Lemma \ref{L.MixLemma}.  Toward contradiction, we suppose that there is a point $x \in T$, a direction $d$ at $x$, and a non-degenerate segment $I \subseteq T$ such that there is no $g \in F_n$ with $gd \cap I$ non-degenerate. 

   Define $X:=\cup I$ where $I$ runs over all non-degenerate arcs of $T$ not meeting $d$ non-degenerately; the collection $\mathscr{X}_1$ of path components of $X$ is a transverse family for the action $F_n \curvearrowright T$.  Note that since for each $\mu \in M_0(T)$, $Supp(\mu)=T$, we have that $T \setminus X$ contains no non-degenerate arc; further, by applying Procedure II of Remark \ref{R.GLProcedure} to $\mathscr{X}_1$ will eventually produce $T$.  By construction, the family $\mathscr{X}$ cannot be a transverse covering of $T$.

   As in Remark \ref{R.GLProcedure}, let $\mathscr{X}_{i+1}', \mathscr{X}_{i+1}$ denote the results of applying Procedure II to $\mathscr{X}_i$.  As $\mathscr{X}$ is not a transverse covering of $T$, some member of $\mathscr{X}_2$ is not closed in $T$.  Let $i$ minimal such that $\mathscr{X}_{i+1}=\{T\}$, then there is $Y \in \mathscr{X}_i'$, which is not closed in $T$, and such that the $F_n$-translates of its closure $\overline{Y}$ give a transverse cover of $T$.  Let $x_1 \in \overline{Y} \setminus Y$, and let $d_1$ denote the unique direction in $\overline{Y}$ at $x_1$.  Note that for any non-degenerate arc $I \subseteq \overline{Y}$, the orbit $F_nx_1$ can meet $I$ only at its endpoints.  For any non-degenerate arc $J \subseteq T$, let $g_1,...,g_r \in F_n$ such that $J \subseteq g_1\overline{Y} \cup ... \cup g_r\overline{Y}$; it follows that $\{y \in J|y=gx_1$ and $gd_1 \cap J$ non-degenerate$\}$ is finite (with cardinality bounded by $2r$), a contradiction.   
  \end{proof}


\end{proof}

\begin{remark}
 One should note that if an action $F_n \curvearrowright T$ has dense orbits and is free, then $T$ is (uniquely) a graph of indecomposable actions.  Indeed, $T$ is not indecomposable if and only if there is a transverse family for the action $F_n \curvearrowright T$; for simplicity, we assume that for any $\mu \in M_0(T)$, $Supp(\mu)=T$.  If $T$ contains a transverse family, then by Procedure II of Remark \ref{R.GLProcedure}, $T$ splits as a graph of actions, say with skeleton $S$.  Since the action $F_n \curvearrowright T$ is free, the action $F_n \curvearrowright S$ has trivial arc stabilizers; \emph{i.e.} $S$ encodes a non-trivial free decomposition of $F_n$, and the claim follows by induction on rank.
\end{remark}

\begin{cor}\label{C.UniqueMaxMix}
 Let $T \in \overline{cv}_n$ have dense orbits, and suppose that for all $\mu \in M_0(T)$, $Supp(\mu)=T$.  There is a unique conjugacy class $[H]$ of finitely generated subgroups $H \leq F_n$ such that:

 \begin{enumerate}
  \item [(i)] the action $H \curvearrowright T_H$ is mixing,
  \item [(ii)] $H = Stab(T_H)$, and
  \item [(iii)] $T_H$ is maximal with respect to (i), (ii).
\end{enumerate}

\end{cor}

\begin{proof}
 Let $T$ as in the statement.  As in the proof of Proposition \ref{P.MixingSubaction}, there are finitely-many (orbits of) directions $\{d_i\}_{i=1,...,k}$ at points $\{x_i\}_{i=1,...,k}$ such that for any non-degenerate arc $I \subseteq T$, there are finitely many elements $g \in F_n$ taking $x_i$ into $I$ such that $gd_i \cap I$ is non-degenerate.  Splitting $T$ apart on the orbits of these directions gives a transverse family $\{gY\}_{g \in F_n}$ such that $Stab(Y) \curvearrowright Y$ is mixing.  The following claim follows easily from the defintion of $Y$.

 \begin{claim}
  Let $Y$ as defined above, and let $Y_I:=\cup J$, where $J$ runs over all non-degenerate arcs contained in $T$ such that there are $g_1,...,g_r \in F_n$ with $J \subseteq g_1I \cup ... \cup g_rI$ and such that $\cup_i g_iI$ is connected.  For any non-degenerate $I \subseteq Y$, $Y_I=Y$.
 \end{claim}

 Now, let $K \curvearrowright X$ be a subaction satisfying (i)-(iii).  Since for all $\mu \in M_0(T)$, one has $Supp(\mu)=T$, it is the case that up to replacing $X$ by a translate and replacing $K$ with a conjugate, we can assume that $Y \cap X$ is non-degenerate.  Further, since $K \curvearrowright X$ is mixing, we can assume that no direction $d_i$ at $x_i$ as in the first paragraph of the proof meets $K$ non-degenerately; indeed, it is immediate that if $d_i$ met $K$ non-degenerately, then $x_i \in \overline{X} \setminus X$, but then the orbit $Kx_i$ would not be dense in every non-degenerate arc of $K$, contradicting mixing.  Let $I \subseteq X \cap Y$ be a non-degenerate arc.  From the claim above, $Y=Y_I$; from maximality of $X$, we can assume that $Y \subseteq X$.  On the other hand, from the fact that $Y=Y_J$ for any non-degenerate arc $J \subseteq Y$, it follows that $Y$ is the maximal mixing subaction of $F_n \curvearrowright T$ containing $I$.  Hence, $X=Y$, and by (ii) $K=Stab(Y)$.  

\end{proof}

\subsection{Decomposing Actions in the Boundary of Outer Space}

In this subsection we collect the preceding results of Section \ref{S.Decompose} to associate to any action $T \in \overline{cv}_n$ a diagram of actions that encodes the structure of $T$.  Let $\leq$ denote the obvious partial order on the finite set $\{Supp(\mu)\}_{\mu \in M_0(T)}$; also denote by $\leq$ the partial order inherited by $M_0(T)$, \emph{i.e.} $\mu \leq \mu'$ if and only if $Supp(\mu) \leq Supp(\mu')$.  Define $[[\mu]]:=\{\nu \in M_0(T) | \nu \leq \mu, \mu \leq \nu \}$; for any $\mu \in  [[\mu_L]]$, we have that the natural map $T \rightarrow T_{\mu}$ is a bijection, and there is an identification $M_0(T)=M_0(T_{\mu})$.  Further, by Lemma \ref{L.ControlL2}, we have that $L^2(T)=L^2(T_{\mu})$.  

Let $T \in \overline{cv}_n$ have dense orbits, and let $\mathscr{F}=\mathscr{F}(T)$ be the transverse family constructed in Subsection \ref{SS.F}.  By Proposition \ref{P.MixingSubaction} and Corollary \ref{C.UniqueMaxMix} we have associated to each $F_n$-orbit $O$ of trees in $\mathscr{F}$ a (canonical) mixing action $H_O \curvearrowright T_{H_O}$, defined up to translation in $T$, \emph{i.e.} up to replacing $H_O$ by a conjugate.  Let $\mu \in [[\mu_L]]$, with $f:T \rightarrow T_{\mu}$ the natural map.  Then it is easy to see that $f(\mathscr{F}):=\{f(Y)|Y \in \mathscr{F}\}=\mathscr{F}(T_{\mu})$ and that $H_O \curvearrowright f(T_{H_O})$ is the mixing action associated to the orbit $f(O)$ in $\mathscr{F}(T_{\mu})$.  

By Lemma \ref{L.SuppGraph} if $\mathscr{F} \neq \{T\}$, then $T$ can be recovered by the iterated graph of actions procedure of Remark \ref{R.GLProcedure} starting from $\mathscr{F}$.  The family $\mathscr{F}$ is canonical, and $Y \in \mathscr{F}$ satisfies the hypotheses of Proposition \ref{P.Diagram}.  Note that the projections of Proposition \ref{P.Diagram} are canonical if we consider them to be defined only up to $[[.]]$.  Hence, we obtain:

\begin{theorem}\label{T.Decompose}
 Let $T \in \overline{cv}_n$ have dense orbits.  With notation as above:
 \begin{enumerate}
  \item [(i)] there is a transverse family $\mathscr{F}$ such that:
    \begin{enumerate}
     \item [(a)] the set of orbits of trees in $\mathscr{F}$ bijectively corresponds to the set of classes $[[\mu]]$ of ergodic measures $\mu \in M_0(T)$ with non-degenerate support,
     \item [(b)] associated to each orbit $O$ of trees in $\mathscr{F}$ is a subtree $T(O)=T([[\mu]]) \subseteq T$, unique up to translation in $T$, such that the action $H(O)=Stab(T(O) \curvearrowright T(O)$ is mixing,
     \item [(c)] the action $F_n \curvearrowright T$ can be recovered via an iterated graphs of actions construction (Remark \ref{R.GLProcedure}, procedure II) starting from $\mathscr{F}$,
     \item [(d)] if $\mathscr{F} \neq \{T\}$, then for any $Y \in \mathscr{F}$, $Stab(Y)$ is a vertex group of a very small splitting of $F_n$.
    \end{enumerate}
  \item [(ii)] there is a diagram of projections of $T$: associated to the class $[[\mu]]$ of an ergodic measure $\mu \in M_0(T)$ with degenerate support is a projection $f_{[[\mu]]}:T \rightarrow T_{[[\mu]]}$ such that:
    \begin{enumerate}
     \item [(a)] $dim(M_0(T_{[[\mu]]})) < dim(M_0(T))$,
     \item [(b)] for ergodic $\mu' \in M_0(T)$ with degenerate support, if $\mu \leq \mu'$, then the projection $f_{[[\mu]]}:T \rightarrow T_{[[\mu]]}$ factors through the projection $f_{[[\mu']]}:T \rightarrow T_{[[\mu']]}$.
     \item [(c)] there is a unique class $[[\nu]]$ of ergodic measures $\nu' \in M_0(T)$ with non-degenerate support such that $\mu \leq \nu$; the subgroup $H([[\nu]])=Stab(T([[\nu]])$ is a point stabilizer in $T_{[[\mu]]}$. 
     \item [(d)] if $[[\mu]]$ is minimal, and if $T_{[[\mu]]}$ contains an exceptional set, the $T$ (and $T_{[[\mu]]}$) splits as a graph of actions.
    \end{enumerate}
 \end{enumerate}
\end{theorem}

\bibliographystyle{amsplain}
\bibliography{indecompREF.bib}
\end{document}